\documentclass[10pt,a4paper]{article}
\usepackage{graphicx}
\usepackage{subcaption}
\usepackage{xcolor}
\usepackage{comment}
\usepackage[utf8]{inputenc}
\usepackage{comment}

\usepackage{hyperref}
\hypersetup{
    colorlinks=true,
    linkcolor=blue,
    citecolor = magenta,
    filecolor=magenta,      
    urlcolor=cyan,
}
\usepackage{siunitx,booktabs}
\usepackage{arydshln}
\usepackage{setspace,amsmath,amsthm,amsfonts,amssymb,mathtools,bm,algorithm,algpseudocode,
longtable,multirow,geometry,mathrsfs,tablefootnote,threeparttable}

\usepackage[font=sf, labelfont={sf,bf}, margin=1cm]{caption}
\newtheorem{Proposition}{Proposition}

\numberwithin{Theorem}{section}
\numberwithin{Definition}{section}
\numberwithin{Lemma}{section}
\numberwithin{Algorithm}{section}
\numberwithin{equation}{section}

\newtheorem{theorem}{Theorem}[section]

\newtheorem{lemma}[theorem]{Lemma} 
\newtheorem{assumption}{Assumption}

\newtheorem{remark}{Remark}
\newcommand\scalemath[2]{\scalebox{#1}{\mbox{\ensuremath{\displaystyle #2}}}}

\makeatletter
\def\@cline#1-#2\@nil{%
  \omit
  \@multicnt#1%
  \advance\@multispan\m@ne
  \ifnum\@multicnt=\@ne\@firstofone{&\omit}\fi
  \@multicnt#2%
  \advance\@multicnt-#1%
  \advance\@multispan\@ne
  \leaders\hrule\@height\arrayrulewidth\hfill
  \cr
  \noalign{\nobreak\vskip-\arrayrulewidth}}
\makeatother
\begin{document}
	\title{An active-set method for sparse approximations\\
	Part I: Separable $\ell_1$ terms}
\author{Spyridon Pougkakiotis \and Jacek Gondzio \and Dionysios S. Kalogerias}

\maketitle

			%\begin{changemargin}{0.8cm}{0.8cm} 
\begin{abstract}
\par In this paper we present an active-set method for the solution of $\ell_1$-regularized convex quadratic optimization problems. It is derived by combining a proximal method of multipliers (PMM) strategy with a standard semismooth Newton method (SSN). The resulting linear systems are solved using a Krylov-subspace method, accelerated by certain general-purpose preconditioners which are shown to be optimal with respect to the proximal parameters. Practical efficiency is further improved by warm-starting the algorithm using a proximal alternating direction method of multipliers. We show that the outer PMM achieves global convergence under mere feasibility assumptions. Under additional standard assumptions, the PMM scheme achieves global linear and local superlinear convergence. The SSN scheme is locally superlinearly convergent, assuming that its associated linear systems are solved accurately enough, and globally convergent under certain additional regularity assumptions. We provide numerical evidence to demonstrate the effectiveness of the approach by comparing it against OSQP and IP-PMM (an ADMM and a regularized IPM solver, respectively) on several elastic-net linear regression and $L^1$-regularized PDE-constrained optimization problems. 
\end{abstract}

\section{Introduction}
\par In this paper we consider convex optimization problems of the following form:
\begin{equation} \label{primal problem} \tag{P}
\underset{x \in \mathbb{R}^n}{\text{min}} \  c^\top x + \frac{1}{2} x^\top Q x + g(x) + \delta_{\mathcal{K}}(x), \qquad \text{s.t.}  \  Ax = b,
\end{equation}
\noindent where $c \in \mathbb{R}^n$, $Q \succeq 0 \in \mathbb{R}^{n\times n}$, $A \in \mathbb{R}^{m \times n}$, $b \in \mathbb{R}^m$, and $g(x) = \|Dx\|_1$, with $D \succeq 0$ and diagonal. Without loss of generality, we assume that $m \leq n$. Furthermore, $\mathcal{K} \coloneqq \{x \in \mathbb{R}^n : a_l \leq x \leq a_u\}$, for some arbitrary (possibly unbounded) vectors $a_l \leq a_u$. Finally, $\delta_{\mathcal{K}}(\cdot)$ is an indicator function for the set $\mathcal{K}$, with $\delta_{\mathcal{K}}^*(\cdot)$ denoting its Fenchel conjugate, that is:
\begin{equation*} 
\delta_{\mathcal{K}}(x) = 
     \begin{cases}
       0, &\quad\text{if } x \in \mathcal{K} \\
       \infty, &\quad\text{otherwise} \\ 
     \end{cases}, \qquad \delta_{\mathcal{K}}^*(x^*) = \underset{x \in \mathbb{R}^n}{\sup} \big\{(x^*)^\top x - \delta_{\mathcal{K}}(x) \big\}.
\end{equation*}
\begin{remark}
We note that problem \textnormal{\eqref{primal problem}} can accommodate instances where sparsity is sought in some appropriate dictionary (i.e. in that case $D$ would be a general rectangular matrix). Indeed, this can be done by appending some additional linear equality constraints in \textnormal{\eqref{primal problem}}, making the $\ell_1$ regularization separable (e.g. see \textnormal{\cite[Sections 3--5]{SIREV:DeSimone_etal}}). However, this would not be a numerically efficient strategy since it would result in additional linear constraints, and the structure of such $\ell_1$ terms would not be fully exploited. Hence, this case is treated separately in an accompanying paper, and is omitted in this work.
\end{remark}
\par Using Fenchel duality, we can easily verify (see Appendix \ref{Appendix: derivation of dual}) that the dual of \eqref{primal problem} is
\begin{equation} \label{dual problem} \tag{D}
\underset{x \in \mathbb{R}^n, y \in \mathbb{R}^m, z \in \mathbb{R}^n}{\text{max}} \  b^\top y - \frac{1}{2} x^\top Q x - \delta^*_{\mathcal{K}}(z) -g^*(A^\top y - c - Qx -z).
\end{equation}
\noindent Throughout the paper we make use of the following blanket assumption.
\begin{assumption} \label{assumption: solution of the QP}
Problems \eqref{primal problem} and \eqref{dual problem} are both feasible.
\end{assumption}
\noindent If $D = 0$, from \cite[Proposition 2.3.4]{BertsekasNedicOzdaglar} we know that Assumption \ref{assumption: solution of the QP} implies that there exists a primal-dual triple $(x^*,y^*,z^*)$ solving \eqref{primal problem}--\eqref{dual problem}. If the primal-dual pair \eqref{primal problem}--\eqref{dual problem} is feasible, it must remain feasible for any $D \succeq 0$, since in this case \eqref{primal problem} can be written as a convex quadratic problem by appending appropriate (necessarily feasible) linear equality and inequality constraints. Thus, Assumption \ref{assumption: solution of the QP} suffices to guarantee that the solution set of \eqref{primal problem}--\eqref{dual problem} is non-empty.
\par There are numerous applications that require the solution of problems of the form of \eqref{primal problem}.  Indeed, \eqref{primal problem} can model linear and convex quadratic programming instances, regularized lasso instances (often arising in signal or image processing and machine learning, e.g. see \cite{SIREV:Chenetal,bookVapnik}), as well as sub-problems arising from the linearization of a nonlinear (possibly nonconvex or nonsmooth) problem (such as those arising within sequential quadratic programming \cite{JCAM:BoggsTolleSQP} or globalized proximal Newton methods \cite{SIAMOPT:Leeetal}). Furthermore, various optimal control problems can be tackled in the form of \eqref{primal problem}, such as those arising from $L^1$-regularized partial differential equation (PDE) optimization, assuming that a \emph{discretize-then-optimize} strategy is adopted (e.g. see \cite{ESAIM:GerdDaniel}). Given the diversity of applications, most of which require a highly-accurate solution, the construction of efficient, scalable, and robust solvers for \eqref{primal problem} has attracted a lot of attention.
\par In particular, there is a plethora of first-order methods capable of finding an approximate solution to \eqref{primal problem}. For example, one could employ proximal (sub-)gradient (e.g. see \cite{SIAM:Beck}) or splitting schemes (e.g. see \cite{SciComp:DengYin}). While such solution methods are very general, easy to implement, and require very little memory, they are usually able to find only an approximate solution, not exceeding 2- or 3-digits of accuracy. If a more accurate solution is needed, then one has to resort to an approach that utilizes second-order information. 
\par There are three major classes of second-order methods for problems of the form of \eqref{primal problem}. Those include globalized (smooth, semismooth, quasi or proximal) Newton methods (e.g. see \cite{MathOR:Han,CAM:MartQi,COAP:Stella_etal}), variants of the proximal point method (e.g. see \cite{COAP:Marchi,IEEE_CDC:Dhingra_etal,MathProgComp:Hermans_etal,SIAMOPT:Leeetal,SIAMOpt:Lietal}), or interior point methods (IPMs) applied to a reformulation of \eqref{primal problem} (e.g. see \cite{SIREV:DeSimone_etal,IPMs:FountoulakisEtAl2013,arXiv:GondPougkPears,NLAA:PearsonPorcStoll}). \par Most globalized Newton-like approaches or proximal point variants studied in the literature are developed for composite programming problems in which either $g(x) = 0$ (e.g. see \cite{NLAA:ChenQi,COAP:Marchi,COAP:GillRobi,MathProg:Itoetal,SIAMOpt:Lietal}) or $\mathcal{K} = \mathbb{R}^n$ (e.g. see \cite{IEEE_CDC:Dhingra_etal,InverseProbs:HansRaasch,SIAMOpt:Lietal2}). More recently there have been developed certain globalized Newton-like schemes, specialized to the case of $L^1$-regularized PDE-constrained optimization (see \cite{OptEng:MannelRund,CAA:Porcellietal}), in which the $\ell_1$ term as well as the box constraints in \eqref{primal problem} are both explicitly handled. We should notice, however, that globalized Newton-like schemes applied to \eqref{primal problem} need additional assumptions on the matrix $Q$, as well as the constraint matrix $A$, since otherwise, the stability of the related Newton linear systems, arising as sub-problems, might be compromised. Under certain assumptions, superlinear convergence of Newton-like schemes is observed ``close to a solution". Under additional assumptions, global convergence can be achieved via appropriate line-search or trust-region strategies (e.g. see the developments in \cite{SIAMOpt:Christofetal,MathOR:Han,MathProg:Itoetal,SIAMOpt:Themelis_etal} and the references therein). 
\par Interior point methods can readily solve problems of the form of \eqref{primal problem} in a polynomial number of steps (\cite{SIREV:DeSimone_etal,IPMs:FountoulakisEtAl2013,arXiv:GondPougkPears,NLAA:PearsonPorcStoll}), and stability of the associated Newton systems can be guaranteed by means of algorithmic regularization (which can be interpreted as the application of a proximal point method, see \cite{OMS:GondAlt,MathProgComp:FriedOrban,COAP:PougkGond}). Nevertheless, the resulting linear systems arising within IPMs are of larger dimensions as compared to those arising within pure Newton-like or proximal approaches, since \eqref{primal problem} needs to be appropriately reformulated into a smooth problem. Furthermore, IPM linear systems have significantly worse conditioning compared to linear systems arising within Newton-like or proximal-Newton methods.
\par The potential stability issues of the linear systems arising within Newton-like schemes can be alleviated by combining Newton-like methods with proximal point variants. Solvers based on the proximal point method can achieve superlinear convergence, assuming their penalty parameters increase indefinitely at a suitable rate (e.g. see \cite{MathOpRes:Rock,SIAMJCO:Rock}). 
For problems \eqref{primal problem}--\eqref{dual problem}, the sub-problems arising within proximal methods are nonsmooth convex optimization instances, and are typically solved by means of semismooth Newton strategies. The resulting linear systems that one has to solve are better conditioned than their (possibly regularized) interior-point counterparts (e.g. see \cite{NLAA:BergGondMartPearPoug,SIREV:DeSimone_etal,IPMs:FountoulakisEtAl2013,arXiv:GondPougkPears,IPMs:WaltzMoralNocedOrban}), however, convergence is expected to be slower, as the method does not enjoy the polynomial worst-case complexity of interior-point methods. Nevertheless, these better conditioned linear systems can in certain cases allow one to achieve better computational and/or memory efficiency, especially if the nonsmooth terms are appropriately handled. 
\par In this paper, we develop an active-set method for \eqref{primal problem}--\eqref{dual problem} by employing an appropriate proximal method of multipliers (PMM) using a standard semismooth Newton (SSN) strategy for solving the associated sub-problems. The SSN linear systems are approximately solved by means of Krylov subspace methods, using certain general-purpose preconditioners. Unlike most proximal point methods given in the literature (e.g. see the primal approaches in \cite{SIAMOPT:Leeetal,IEEE_DC:Patrinos_etal}, the dual approaches in \cite{SIAMOpt:Lietal2,SIAMOpt:Lietal} or the primal-dual approaches in \cite{COAP:Marchi,COAP:GillRobi,MathOpRes:Rock}), the proposed method introduces  proximal terms for each primal and dual variable of the problem, and this results in Newton linear systems which are easy to precondition and solve. We explicitly deal with each of the two nonsmooth terms of the objective in \eqref{primal problem}, which contributes to the simplification of the resulting SSN linear systems, and paves the way for generalizing this approach to a wider class of problems. In an accompanying paper we discuss the extension of the proposed method to problems with general piecewise-linear terms in the objective, in a way that allows for full exploitation of the piecewise-linear structure. 
\par We show that global convergence of the outer PMM scheme is guaranteed with the minimal assumption
of primal and dual feasibility, while global linear and local superlinear convergence holds under standard additional assumptions. Furthermore, we note that while most proximal Newton-like methods proposed in the literature allow inexactness in the solution of the associated Newton linear systems, the development of general-purpose preconditioners for them is lacking. Indeed, aside from the work in \cite{CAA:Porcellietal} which is specialized to the case of $L^1$-regularized PDE constrained optimization, most proximal Newton-like schemes utilizing Krylov subspace methods do so without employing any preconditioner (e.g. see \cite{NLAA:ChenQi,SIAMOpt:Lietal2,IEEE_DC:Patrinos_etal}). Drawing from the interior point literature, and by suitably specializing a preconditioning approach given in \cite{NLAA:BergGondMartPearPoug}, we propose general-purpose positive definite preconditioners that are robust with respect to the penalty parameters of the PMM. The positive definiteness of the preconditioners allows the use of (memory efficient) symmetric Krylov subspace solvers such as the minimum residual method (MINRES), \cite{PaigeSaundersSIAMNumAnal}.
\par The method deals with general box constraints and thus there is no need for introducing auxiliary variables to deal with upper and lower bounds separately, something that is required when employing conic-based solvers. As a result, the associated linear systems solved within SSN have significantly smaller dimensions, compared to linear systems arising within interior point methods suitable for the solution of $\ell_1$-regularized convex quadratic problems (e.g. see \cite{SIREV:DeSimone_etal,arXiv:GondPougkPears,NLAA:PearsonPorcStoll}), potentially making the proposed approach a more attractive alternative for large-scale instances. Finally, the method is easily warm-started using a proximal alternating direction method of multipliers to further improve its efficiency at a low computational cost. We provide numerical evidence to demonstrate that the proposed active-set scheme is efficient and robust when applied to elastic-net linear regression and $\ell_1$-regularized problems arising from PDE-constrained optimization. We compare the proposed algorithm against IP-PMM (see \cite{COAP:PougkGond}), which is a robust regularized interior point method utilized in \cite{SIREV:DeSimone_etal,arXiv:GondPougkPears}, as well as the well-known (ADMM-based) OSQP method \cite{osqp}, and numerically showcase certain benefits of the active-set method. Indeed, we observe that the linear systems arising within the proposed scheme are much better conditioned compared to those arising within IPMs, allowing for the solution of several instances at a significantly lower CPU time. Additionally, we demonstrate that the use of preconditioning and the structure exploitation of the $\ell_1$ terms allow the proposed active-set scheme to outscale and outperform OSQP and IP-PMM on large or numerically challenging instances.
\par To summarize, in Section \ref{sec: PAL penalties} we derive a proximal method of multipliers and discuss its convergence properties. Then, in Section \ref{sec: SSN method} we present a well-studied locally superlinearly convergent semismooth Newton scheme used to approximately solve the PMM sub-problems, noticing that its global convergence can be shown under additional regularity assumptions. Furthermore, we propose general-purpose preconditioners for the associated SSN linear systems and analyze their effectiveness. In Section \ref{sec: warm start} we present a warm-starting strategy for the method. Then, in Section \ref{sec: applications}, the overall approach is extensively tested on certain linear regression and partial differential equation constrained optimization problems. Finally, we derive some conclusions in Section \ref{sec: Conclusions}.

\paragraph{Notation} Given a vector $x$ in $\mathbb{R}^n$, $\|x\|$ denotes the Euclidean norm. Letting $R \succ 0$ be a symmetric positive definite matrix, we denote $\|x\|_R^2 = x^\top R x$. Given a closed set $\mathcal{K} \subset \mathbb{R}^n$, we write $\Pi_{\mathcal{K}}(x) \coloneqq \arg\min\{\|x-z\| \colon z \in \mathcal{K}\}$, while for any $R \succ 0$ we write $\textnormal{dist}_R(z,\mathcal{K}) \coloneqq \inf_{z'\in \mathcal{K}}\|z-z'\|_R$. If $R = I$, we assume that $\textnormal{dist}_I(z,\mathcal{K}) \equiv \textnormal{dist}(z,\mathcal{K})$. Given an arbitrary rectangular matrix $A$, $\sigma_{\max}(A)$ denotes its maximum singular value. For an arbitrary square matrix $B$, $\lambda(B)$ is the set of eigenvalues of $B$ while $\lambda_{\max}(B)$ (resp. $\lambda_{\min}(B)$) denotes its maximum (resp. minimum) eigenvalue. Given an index set  $\mathcal{D}$, $|\mathcal{D}|$ denotes its cardinality. Given a rectangular matrix $A \in \mathbb{R}^{m \times n}$ and an index set $\mathcal{B} \subseteq \{1,\ldots,n\}$, we denote the columns of $A$, the indices of which belong to $\mathcal{B}$, as $A_{\mathcal{B}}$. Given a square matrix $Q \in \mathbb{R}^{n\times n}$, we denote the subset of columns and rows of $Q$, the indices of which belong to $\mathcal{B}$, as $Q_{\left(\mathcal{B},\mathcal{B}\right)}$. We denote by $\textnormal{Diag}(Q)$ the diagonal matrix with diagonal elements equal to those of $Q$, and by $\textnormal{Off}(Q)$ the square matrix with off-diagonal elements equal to those of $Q$ and zero diagonal.
\section{A primal-dual proximal method of multipliers} \label{sec: PAL penalties}
\par In what follows, we derive the proximal augmented Lagrangian penalty function corresponding to the primal problem \eqref{primal problem}. Using the latter, we derive a primal-dual PMM for solving the pair \eqref{primal problem}--\eqref{dual problem}. The convergence of this PMM scheme is subsequently analyzed, assuming that we are able to find sufficiently accurate solutions to its associated sub-problems. In the next section, we briefly present a standard semismooth Newton scheme for the solution of these sub-problems. 
\subsection{Derivation of the outer method}
\par We begin by deriving the Lagrangian associated to \eqref{primal problem}. First, we define the function $\varphi(x) \coloneqq c^\top x + \frac{1}{2} x^\top Q x + g(x) + \delta_{\mathcal{K}}(x) + \delta_{\{0\}}(b-Ax)$. Following the dualization strategy proposed in \cite[Chapter 11]{Springer:RockWets}, we let
$\hat{\varphi}(x,u',w') \coloneqq c^\top x + \frac{1}{2} x^\top Q x +  g(x) + \delta_{\mathcal{K}}(x+w') + \delta_{\{0\}}(b-Ax+ u')$, for which it holds that $\varphi(x) =\hat{\varphi}(x,0,0)$. Then, the Lagrangian associated to \eqref{primal problem} reads:
\begin{equation*}
\begin{split}
\ell (x,y,z) &\coloneqq\ \inf_{u',w'} \bigg\{\hat{\varphi}(x,u',w') - y^\top  u' - z^\top  w'\bigg\} \\ &=\ c^\top x + \frac{1}{2} x^\top Q x + g(x)  - \sup_{w'} \bigg\{z^\top  w' - \delta_{\mathcal{K}}(x+w') \bigg\} - \sup_{u'} \bigg\{y^\top  u' - \delta_{\{0\}}(b-Ax + u') \bigg\} \\
&=\ c^\top x + \frac{1}{2} x^\top Q x + g(x) +  z^\top  x - \delta^*_{\mathcal{K}}(z) - y^\top (Ax-b),
\end{split}
\end{equation*}
\noindent where we used the definition of the Fenchel conjugate. Before deriving the augmented Lagrangian associated to \eqref{primal problem}, we introduce some necessary notation as well as relations that will be used later on. Firstly, given a convex function $p \colon \mathbb{R}^n \mapsto \mathbb{R}$, we define
\[ \textbf{prox}_p (u) \coloneqq \arg\min_x \bigg\{ p(x) + \frac{1}{2}\|u-x\|^2\bigg\}.\] 
\noindent Then, given some positive constant $\beta$, it holds that (\emph{Moreau Identity}, see \cite{BSMF:Moreau}):
\begin{equation} \label{Moreau Identity}
\textbf{prox}_{\beta p}(u') + \beta \textbf{prox}_{\beta^{-1}p^*}(\beta^{-1}u') = u'.
\end{equation}
\noindent Finally,  we have that (e.g. see \cite[Equation 2.2]{ApplMathOpt:HirStroNgu})
\begin{equation} \label{property of prox of f and its conjugate}
\frac{1}{2}\|\textbf{prox}_p (x)\|^2 + p^*(\textbf{prox}_{p^*}(x)) = \frac{1}{2}\bigg[\|x\|^2 - \|x-\textbf{prox}_{p}(x)\|^2 \bigg] - p(\textbf{prox}_{p}(x)).
\end{equation}
\par Given a penalty parameter $\beta > 0$, the augmented Lagrangian corresponding to \eqref{primal problem} reads:
\begin{equation} \label{augmented lagrangian of the primal 1}
\begin{split}
\mathcal{L}_{\beta}(x; y, z) &\coloneqq\ \sup_{u',w'} \bigg\{\ell(x,u',w') - \frac{1}{2\beta}\|u'-y\|^2 - \frac{1}{2\beta}\|w'-z\|^2 \bigg\} \\
&=\ c^\top x + \frac{1}{2} x^\top Q x + g(x) - \inf_{u'} \bigg\{u'^\top (Ax-b) + \frac{1}{2\beta}\|u'-y\|^2 \bigg\} \\&\quad \ - \inf_{w'} \bigg\{-w'^\top x + \delta^*_{\mathcal{K}}(w') + \frac{1}{2\beta}\|w'-z\|^2 \bigg\} \\
&=\ \scalemath{0.97}{c^\top x + \frac{1}{2} x^\top Q x + g(x) - y^\top (Ax-b) + \frac{\beta}{2}\|Ax-b\|^2 + x^\top \big(\textbf{prox}_{\beta \delta^*_{\mathcal{K}}}(z+ \beta x)\big)}\\ &\quad \  - \delta^*_{\mathcal{K}}\big( \textbf{prox}_{\beta \delta^*_{\mathcal{K}}}(z+ \beta x)\big) - \frac{1}{2\beta}\|\textbf{prox}_{\beta \delta^*_{\mathcal{K}}}(z+ \beta x) - z\|^2,
\end{split}
\end{equation}
\noindent where we used the fact that if $p_1(u') = p_2(u') + r^\top x$, where $p_1(\cdot)$ and $p_2(\cdot)$ are two closed convex functions, and $r$ is a vector, then $\textbf{prox}_{p_1}( u') = \textbf{prox}_{p_2}(u'-r)$. Using \eqref{Moreau Identity}, we obtain:
\begin{equation} \label{augmented lagrangian of the primal 2}
\begin{split}
\mathcal{L}_{\beta}(x; y, z) &= \scalemath{1}{\ c^\top x + \frac{1}{2} x^\top Q x + g(x) - y^\top (Ax-b) + \frac{\beta}{2}\|Ax-b\|^2} \\ &\qquad \ \scalemath{1}{+\ x^\top \big((z+\beta x) - \beta \textbf{prox}_{\beta^{-1} \delta_{\mathcal{K}}}(\beta^{-1}z+  x)\big) - \delta^*_{\mathcal{K}}\big( \textbf{prox}_{\beta \delta^*_{\mathcal{K}}}(z+ \beta x)\big)}\\ &\qquad \  \scalemath{1}{- \frac{1}{2\beta}\|\beta x - \beta\textbf{prox}_{\beta^{-1} \delta_{\mathcal{K}}}(\beta^{-1}z+  x)\|^2}\\
&= \ \scalemath{1}{c^\top x + \frac{1}{2} x^\top Q x + g(x)  - y^\top (Ax-b) + \frac{\beta}{2}\|Ax-b\|^2  + x^\top \big(z+\frac{\beta}{2} x\big)}\\ &\qquad \ \scalemath{0.97}{- \bigg(\delta^*_{\mathcal{K}}\big( \textbf{prox}_{\beta \delta^*_{\mathcal{K}}}(z+ \beta x)\big)  + \frac{1}{2\beta}\| \beta\textbf{prox}_{\beta^{-1} \delta_{\mathcal{K}}}(\beta^{-1}z+  x)\|^2\bigg).}
\end{split}
\end{equation}
\noindent Finally, we write:
\begin{equation*}
\begin{split}
R \coloneqq &\ \scalemath{1}{\bigg(\delta^*_{\mathcal{K}}\big( \textbf{prox}_{\beta \delta^*_{\mathcal{K}}}(z+ \beta x)\big)  + \frac{1}{2\beta}\| \beta\textbf{prox}_{\beta^{-1} \delta_{\mathcal{K}}}(\beta^{-1}z+  x)\|^2\bigg)}  \\
=&\ \scalemath{1}{\frac{1}{\beta}\Bigg(\frac{1}{2}\bigg(\|z+\beta x\|^2 - \|z + \beta x - \beta \textbf{prox}_{\beta^{-1}\delta_{\mathcal{K}}}(\beta^{-1}z + x)\|^2 \bigg)-\beta \delta_{\mathcal{K}} \big(\textbf{prox}_{\beta^{-1}\delta_{\mathcal{K}}}(\beta^{-1}z + x) \big) \Bigg)}\\
=&\ \scalemath{0.97}{\frac{1}{2\beta}\bigg(\|z\|^2 + \beta^2 \|x\|^2 + 2\beta z^\top x -\|z + \beta x - \beta \Pi_{\mathcal{K}}(\beta^{-1}z + x)\|^2  \bigg)},
\end{split}
\end{equation*}
\noindent where we used \eqref{property of prox of f and its conjugate}, along with the fact that $(\beta \delta^*_{\mathcal{K}})^*(x) = \beta\delta_{\mathcal{K}}(\beta^{-1}x)$, while $\textbf{prox}_{(\beta \delta^*_{\mathcal{K}})^*}(x) = \beta \textbf{prox}_{\beta^{-1}\delta_{\mathcal{K}}}(\beta^{-1}x)$. Substituting $R$ in the last line of \eqref{augmented lagrangian of the primal 2}, yields:
\begin{equation} \label{final augmented lagrangian of the primal}
\begin{split}
\mathcal{L}_{\beta}(x; y, z) &= \ c^\top x + \frac{1}{2} x^\top Q x + g(x) - y^\top (Ax-b) + \frac{\beta}{2}\|Ax-b\|^2  \\ &\qquad \ - \frac{1}{2\beta}\|z\|^2  + \frac{1}{2\beta}  \|z + \beta x - \beta \Pi_{\mathcal{K}}(\beta^{-1}z + x)\|^2. 
\end{split}
\end{equation}
\par Assume that at iteration $k\geq 0 $ we have the estimates $(x_k,y_k,z_k)$ as well as the penalty parameters $\beta_k,\ \rho_k$, such that $\rho_k \coloneqq \frac{\beta_k}{\tau_k}$, where $\{\tau_k\}_{k = 0}^{\infty}$ is a non-increasing positive sequence, i.e. $\tau_k > 0$ for all $k \geq 0$. We define the continuously differentiable function $\phi(x) \equiv\  \phi_{\rho_k,\beta_k}(x;x_k,y_k,z_k) \coloneqq  \mathcal{L}_{\beta_k}(x;y_k,z_k)  - g(x) + \frac{1}{2\rho_k}\|x-x_k\|^2.$ Using the previous notation, we need to find $x^*$ such that
\[\left(\nabla \phi(x^*)\right)^\top (x - x^*) + g(x) - g(x^*) \geq 0,\qquad \forall\ x \in \mathbb{R}^n,\textnormal{ where}\]
\[ \nabla \phi(x) = c +  Q x - A^\top y_k + \beta_k A^\top  (Ax-b) + (z_k + \beta_k x) - \beta_k \Pi_{\mathcal{K}}(\beta_k^{-1}z_k + x) + \rho_k^{-1}(x-x_k).\]
\noindent Let $y = y_k - \beta_k (Ax-b)$ and write the optimality conditions of $\underset{x}{\min}\ \psi(x) \coloneqq \phi(x)+ g(x)$ as
\begin{equation} \label{PD-SSN-PMM Proximal Optimality Conditions}
(0,0) \in  F_{\beta_k,\rho_k}(x,y) \coloneqq   \left\{(u',v') \colon u' \in r_{\beta_k,\rho_k}(x,y) + \partial g(x),\quad v' = Ax + \beta_k^{-1}(y-y_k)-b \right\},
\end{equation}
\noindent where $r_{\beta_k,\rho_k}(x,y) \coloneqq   c + Q x - A^\top  y + (z_k + \beta_k x) - \beta_k \Pi_{\mathcal{K}}(\beta_k^{-1}z_k + x) + \rho_k^{-1}(x-x_k).$
\noindent We now describe the primal-dual PMM in Algorithm \ref{primal-dual PMM algorithm}.
\renewcommand{\thealgorithm}{PD-PMM}

\begin{algorithm}[!ht]
\caption{Primal-dual proximal method of mutlipliers}
    \label{primal-dual PMM algorithm}
    \textbf{Input:}  $(x_0,y_0,z_0) \in \mathbb{R}^n \times \mathbb{R}^m \times \mathbb{R}^n$, $\beta_0,\ \beta_{\infty},\ \tau_{\infty} > 0$, $\{\tau_k\}_{k=0}^{\infty}$ such that $\tau_k \searrow \tau_{\infty} > 0$.
\begin{algorithmic}
\State Choose a sequence of positive numbers $\{\epsilon_k\}$ such that $\epsilon_k \rightarrow 0$. 
\For {($k = 0,1,2,\ldots$)}
\State Find $(x_{k+1},y_{k+1})$ such that:
\begin{equation} \label{primal-dual PMM main sub-problem}
\textnormal{dist}\left(0,F_{\beta_k,\rho_k}\left(x_{k+1},y_{k+1}\right)\right)  \leq \epsilon_k,
\end{equation}
\State where, letting $\hat{r} = r_{\beta_k,\rho_k}(x_{k+1},y_{k+1})$, we have
\[\textnormal{dist}\left(0,F_{\beta_k,\rho_k}(x_{k+1},y_{k+1})\right) = \left\|\begin{bmatrix}\hat{r}+ \Pi_{\partial\left(g\left(x_{k+1}\right)\right)}\left(-\hat{r}\right)\\
Ax_{k+1} + \beta_k^{-1}(y_{k+1}-y_k) - b
\end{bmatrix}\right\|. \]
   \begin{flalign}  \label{primal-dual PMM z-update}
\ \ \ \ z_{k+1} &= \  (z_k + \beta_k x_{k+1}) - \beta_k\Pi_{\mathcal{K}}\big(\beta_k^{-1}z_k + x_{k+1} \big).&&
\end{flalign} 
\begin{flalign} \ \ \ \ \beta_{k+1} &\nearrow \beta_{\infty} \leq \infty, \quad \rho_{k+1} = \frac{\beta_{k+1}}{\tau_{k+1}}.&&
\end{flalign}
\EndFor
\State \Return $(x_k,y_k,z_k)$.
\end{algorithmic}
\end{algorithm}
\par Notice that we allow step \eqref{primal-dual PMM main sub-problem} to be computed inexactly. In Section \ref{subsection: PMM convergence analysis} we will provide precise conditions on the error sequence guaranteeing that Algorithm \ref{primal-dual PMM algorithm} achieves global convergence, and additional conditions for achieving a local linear or superlinear rate (where the local superlinear convergence requires that $\beta_k \rightarrow \infty$). Further conditions on the starting point and on the starting penalty parameter $\beta_0$, required to guarantee a global linear convergence rate, are also discussed.  At this point, we note that the characterization of $\textnormal{dist}\left(0,F_{\beta_k,\rho_k}(x,y)\right)$ follows from the definition of $F_{\beta_k,\rho_k}(x,y)$ as well as from the definition of $\textnormal{dist}(x,\mathcal{A})$ for some closed convex set $\mathcal{A}$.  Finally, we observe that the condition in \eqref{primal-dual PMM main sub-problem} can be evaluated expeditiously, since $g(x) = \|Dx\|_1$ for some diagonal matrix $D \succeq 0$ (i.e. its subdifferential is explicitly known).

\subsection{Convergence analysis} \label{subsection: PMM convergence analysis}
\par In this section we provide conditions on the error sequence $\{\epsilon_k\}$ in \eqref{primal-dual PMM main sub-problem} that guarantee the convergence of Algorithm \ref{primal-dual PMM algorithm}, potentially at a global linear or local superlinear rate. The analysis is based on \cite[Section 2]{SIAMOpt:Lietal} (or by an extension of the analyses in \cite{MathOpRes:Rock,SIAMJCO:Rock}) after connecting Algorithm \ref{primal-dual PMM algorithm} to an appropriate proximal point iteration. First, we define the maximal monotone operator $T_{\ell} \colon \mathbb{R}^{2n+m} \rightrightarrows \mathbb{R}^{2n+m}$, associated to \eqref{primal problem}--\eqref{dual problem}:
\begin{equation} \label{maximal monotone operator associated to primal-dual problem}
\begin{split}
T_{\ell}(x,y,z) &\coloneqq \bigg\{(u',v',w') \colon  v' \in  Qx + c - A^\top y + z + \partial g(x),\  u' = Ax - b,\  w'+ x \in \partial\delta^*_{\mathcal{K}}(z)\bigg\}\\
& = \bigg\{(u',v',w') \colon  v' \in Qx + c - A^\top y + z + \partial g(x),\   u' = Ax - b,\  z \in \partial\delta_{\mathcal{K}}(x+w')\bigg\}.
\end{split}
\end{equation}
\noindent  The inverse of this operator reads
\begin{equation} \label{inverse of maximal monotone operator associated to primal-dual problem}
\begin{split}
T^{-1}_{\ell}(u',v',w') &\coloneqq \arg\max_{y,z}\min_x\left\{\ell(x,y,z) + u'^\top x - v'^\top y -w'^\top z\right\}.
\end{split}
\end{equation}
\noindent Notice that Assumption \ref{assumption: solution of the QP} implies that $T^{-1}_{\ell}(0) \neq \emptyset$. Following the result in \cite{SIAMOpt:Lietal}, we note that $T_{\ell}$ is in fact a polyhedral multifunction (see \cite{RobinsonMathProgStud} for a detailed discussion on the properties of such multifunctions). In light of this property of $T_{\ell}$ we note that the following specialized \emph{metric subregularity} condition (see \cite{Springer:DonRock} for a definition) holds automatically without any additional assumptions.
\begin{lemma}\label{lemma: error condition for polyhedral multifucntion}
For any $r > 0$, there exists $\kappa > 0$ such that
\begin{equation} \label{eqn: error condition for polyhedral multifunction}
\textnormal{dist}\big(p,T_{\ell}^{-1}(0)\big) \leq \kappa\ \textnormal{dist}\big(0,T_{\ell}(p)\big),\quad \forall\ p \in \mathbb{R}^{2n+m},\ \textnormal{with }\textnormal{dist}\big(p,T_{\ell}^{-1}(0)\big) \leq r.
\end{equation}
\end{lemma}
\begin{proof}
\noindent The reader is referred to \cite[Lemma 2.4]{SIAMOpt:Lietal} as well as \cite{RobinsonMathProgStud}.
\end{proof}
\par Next, let some sequence of positive definite matrices $\{R_k\}_{k=0}^{\infty}$ with $R_k \coloneqq \tau_k I_n \oplus I_m \oplus I_n$, for all $k \geq 0$, where $\tau_k$ is defined in Algorithm \ref{primal-dual PMM algorithm} and $\oplus$ denotes the direct sum of two matrices. We define the single-valued proximal operator $P_k \colon \mathbb{R}^{2n + m} \mapsto \mathbb{R}^{2n+m}$, associated to \eqref{maximal monotone operator associated to primal-dual problem}:
\begin{equation} \label{proximal operator}
P_k \coloneqq \big(R_k + \beta_k T_{\ell}\big)^{-1} R_k.
\end{equation}
\noindent In particular, under our assumptions on the matrices $R_k$, we have that (e.g. see \cite{SIAMJCO:Rock}) for all $(u_1,v_1,w_1),\ (u_2,v_2,w_2) \in \mathbb{R}^{2n+m}$, the following inequality (non-expansiveness) holds
\begin{equation} \label{non-expansiveness of P_k}
\left\|(u_1,v_1,w_1)-P_k(u_2,v_2,w_2)\right\|_{R_k} \leq \left\|(u_1,v_1,w_1)-(u_2,v_2,w_2)\right\|_{R_k}. \end{equation}
\noindent Obviously, we can observe that if $(x^*,y^*,z^*) \in T_{\ell}^{-1}(0)$, then $P_k(x^*,y^*,z^*) = (x^*,y^*,z^*)$. We are now able to connect Algorithm \ref{primal-dual PMM algorithm} with the proximal point iteration produced by \eqref{proximal operator}.
\begin{Proposition} \label{proposition: connection to PMM}
Let $\{(x_k,y_k,z_k)\}_{k=0}^{\infty}$ be a sequence of iterates produced by Algorithm \textnormal{\ref{primal-dual PMM algorithm}}. Then, for every $k \geq 0$ we have that
\begin{equation} \label{connection of PMM with PPA error}
\left\|(x_{k+1},y_{k+1},z_{k+1})-P_k(x_k,y_k,z_k)\right\|_{R_k} \leq \frac{\beta_k}{\min\{\sqrt{\tau_k},1\}}\textnormal{dist}\left(0,F_{\beta_k,\rho_k}\left(x_{k+1},y_{k+1}\right)\right).
\end{equation}
\end{Proposition}
\begin{proof}
\noindent Firstly, let us define the pair 
\[\scalemath{0.95}{(\hat{u},\hat{v}) \coloneqq  \left(r_{\beta_k,\rho_k}\left(x_{k+1},y_{k+1}\right) + \Pi_{\partial g(x_{k+1})}\left(-r_{\beta_k,\rho_k}\left(x_{k+1},y_{k+1}\right) \right),Ax_{k+1} + \beta_k^{-1}\left(y_{k+1}-y_{k}\right)-b\right).}\]
\noindent We observe that given a sequence produced by Algorithm \ref{primal-dual PMM algorithm}, we have
\begin{equation}\label{connection of PPM with PPA eqn}
\left(\hat{u},\hat{v},0\right)  +
\beta_k^{-1}\left( \tau_k(x_k - x_{k+1}),y_k - y_{k+1},z_k - z_{k+1}\right) \in\ T_{\ell}(x_{k+1},y_{k+1},z_{k+1}).
\end{equation}
\noindent To show this, we firstly notice that
\begin{equation*} 
\begin{bmatrix} \hat{u} \\
\hat{v} \end{bmatrix} +  
\beta_k^{-1}\begin{bmatrix} \tau_k(x_k - x_{k+1})\\
 y_k - y_{k+1}
 \end{bmatrix} \in \begin{bmatrix}
 Qx_{k+1} + c - A^\top y_{k+1} + z_{k+1} + \partial g(x_{k+1})\\
 Ax_{k+1} - b
 \end{bmatrix},
 \end{equation*}
\noindent where we used the definition of the $(\hat{u},\hat{v})$ as well as \eqref{primal-dual PMM z-update}. It remains to show that $\beta_k^{-1}(z_k - z_{k+1}) \in - x_{k+1} + \partial \delta_{\mathcal{K}}^*(z_{k+1})$. Alternatively, from the second equality in \eqref{maximal monotone operator associated to primal-dual problem}, we need to show that $z_{k+1} \in \partial \delta_{\mathcal{K}}\big(x_{k+1} + \beta_k^{-1}(z_k - z_{k+1})\big)$. To that end, we characterize the subdifferential of $\partial \delta_{\mathcal{K}}(\cdot)$. By convention we have that $\partial \delta_{\mathcal{K}}(\tilde{x}) = \emptyset$ if $\tilde{x} \notin \mathcal{K}$. Hence, assume that $\tilde{x} \in \mathcal{K}$. Then, we obtain
\[ \partial \delta_{\mathcal{K}}(\tilde{x}) = \big\{\tilde{z} \in \mathbb{R}^n \colon \tilde{z}^\top (\hat{x} - \tilde{x}) \leq \delta_{\mathcal{K}}(\hat{x}),\ \forall\ \hat{x} \in \mathbb{R}^n\big\}.\]
\noindent By inspection, we fully characterize the latter component-wise, for any $i \in \{1,\ldots,n\}$, as follows
\begin{equation*}
\partial \delta_{[l_i,u_i]}(\tilde{x}_i) = \begin{cases} 
      \{0\} & \tilde{x}_i \in (l_i,u_i), \\
      (-\infty,0] & \tilde{x}_i = l_i, \\
      [0,\infty) & \tilde{x}_i = u_i. 
   \end{cases}
\end{equation*}
\noindent From \eqref{primal-dual PMM z-update} we have that $z_{k+1} = z_k + \beta_k x_{k+1} - \beta_k \Pi_{\mathcal{K}}(\beta_k^{-1}z_k + x_{k+1})$. Proceeding component-wise, if $(\beta_k^{-1} z_{k,i} + x_{k+1,i}) \in (l_i,u_i)$, then $z_{k+1,i} = 0$, i.e. 
\begin{equation*}
0=z_{k+1,i} \in \partial \delta_{[l_i,u_i]}\big(x_{k+1,i} + \beta_k^{-1}(z_{k,i} + z_{k+1,i})\big) = \partial \delta_{[l_i,u_i]}\big(\beta_k^{-1}z_{k,i} + x_{k+1,i}\big).
\end{equation*}
\noindent If $(\beta_k^{-1} z_{k,i} + x_{k+1,i}) \leq l_i$, then $z_{k+1,i} \leq 0$, and from the previous characterization we obtain that $z_{k+1,i} \in \partial\delta_{\mathcal{K}}(l_i\big)_i$. Finally, if $(\beta_k^{-1} z_{k,i} + x_{k+1,i}) \geq u_i$, we obtain that $z_{k+1,i} \geq 0$ and thus $z_{k+1,i} \in \partial\delta_{\mathcal{K}}(u_i\big)_i$. This shows that \eqref{connection of PPM with PPA eqn} holds. Next, by appropriately re-arranging \eqref{connection of PPM with PPA eqn} we obtain
\[\left(x_{k+1}, y_{k+1}, z_{k+1}\right)= P_k \left( R_k^{-1}\left(\hat{u},
\hat{v},0\right)  +
\left( x_k,y_k,z_k \right)\right),\]
\noindent where $P_k$ is defined in \eqref{proximal operator}. Subtracting both sides by $P_k(x_k,y_k,z_k)$, taking norms, using the non-expansiveness of $P_k$ (see \eqref{non-expansiveness of P_k}), and noting that $\textnormal{dist}\left(0,F_{\beta_k,\rho_k}\left(x_{k+1},y_{k+1}\right)\right) = \|(\hat{u},\hat{v})\|$, yields \eqref{connection of PMM with PPA error} and concludes the proof.
\end{proof}
\par Now that we have established the connection of Algorithm \ref{primal-dual PMM algorithm} with the proximal point iteration governed by the operator $P_k$ defined in \eqref{proximal operator}, we can directly provide conditions on the error sequence in \eqref{primal-dual PMM main sub-problem}, to guarantee global (possibly linear) and local linear (potentially superlinear) convergence of Algorithm \ref{primal-dual PMM algorithm}. To that end, we will make use of certain results, as reported in \cite[Section 2]{SIAMOpt:Lietal}. Firstly, we provide the global convergence result for the algorithm.
\begin{theorem}
Let Assumption \textnormal{\ref{assumption: solution of the QP}} hold. Let $\{(x_k,y_k,z_k)\}_{k=0}^{\infty}$ be generated by Algorithm \textnormal{\ref{primal-dual PMM algorithm}}. Furthermore, assume that we choose a sequence $\{\epsilon_k\}_{k=0}^{\infty}$ in \eqref{primal-dual PMM main sub-problem}, such that
\begin{equation} \label{eqn: error condition for global convergence}
\epsilon_k \leq \frac{\min\{\sqrt{\tau_k},1\}}{\beta_k}\delta_k,\quad 0 \leq \delta_k,\quad \sum_{k=0}^{\infty} \delta_k < \infty.
\end{equation}
\noindent Then, $\{(x_k,y_k,z_k)\}_{k=0}^{\infty}$ is bounded and converges to a primal-dual solution of \eqref{primal problem}--\eqref{dual problem}.
\end{theorem}
\begin{proof}
\noindent The proof is omitted since it is a direct application of \cite[Theorem 2.3]{SIAMOpt:Lietal}. 
\end{proof}
\par Next, we discuss local linear (and potentially superlinear) convergence of Algorithm \ref{primal-dual PMM algorithm}. To that end, let $r > \sum_{k = 0}^{\infty} \delta_k$, where $\delta_k$ is defined in \eqref{eqn: error condition for global convergence}. Then, from Lemma \ref{lemma: error condition for polyhedral multifucntion} we know that there exists $\kappa > 0$ associated with $r$ such that
\begin{equation} \label{eqn: specific kappa and r for error condition}
 \textnormal{dist}\big((x,y,z),T_{\ell}^{-1}(0)\big) \leq \kappa\ \textnormal{dist}\big(0,T_{\ell}(x,y,z)\big),
 \end{equation}
\noindent for all $(x,y,z) \in \mathbb{R}^{2n+m}$ such that $\textnormal{dist}\big((x,y,z),T_{\ell}^{-1}(0)\big) \leq r$. 
\begin{theorem} \label{thm:local convergence} Let Assumption \textnormal{\ref{assumption: solution of the QP}} hold, and assume that $(x_0,y_0,z_0)$ is chosen to satisfy \[\textnormal{dist}_{R_0}\big((x_0,y_0,z_0),T_{\ell}^{-1}(0)\big) \leq r - \sum_{k = 0}^{\infty} \delta_k,\]
\noindent where $\{\delta_k\}_{k=0}^{\infty}$ is given in \eqref{eqn: error condition for global convergence}. Let also $\kappa$ be given as in \eqref{eqn: specific kappa and r for error condition} and assume that we choose a sequence $\{\epsilon_k\}_{k=0}^{\infty}$ in \eqref{primal-dual PMM main sub-problem} such that
\begin{equation} \label{eqn: error condition for superlinear convergence}
\epsilon_k \leq \frac{\min\{\sqrt{\tau_k},1\}}{\beta_k} \min\big\{\delta_k, \delta_k' \|(x_{k+1},y_{k+1},z_{k+1}) - (x_k,y_k,z_k)\|_{R_k}\big\},
\end{equation}
\noindent where $0 \leq \delta_k$, $\sum_{k=0}^{\infty} \delta_k < \infty$, and $0 \leq \delta_k' < 1$, $\sum_{k = 0}^{\infty} \delta_k' < \infty$. Then, for all $k \geq 0$ we have that
\begin{equation} \label{eqn: superlinear convergence error decay}
\begin{split}
\textnormal{dist}_{R_{k+1}}\left((x_{k+1},y_{k+1},z_{k+1},T_{\ell}^{-1}(0)\right) & \leq \mu_k \textnormal{dist}_{R_k}\left( x_{k},y_{k},z_{k},T_{\ell}^{-1}(0)\right),
\end{split}
\end{equation}
\[\mu_k \coloneqq (1-\delta_k')^{-1}\left(\delta_k' + (1+\delta_k')\frac{\kappa \gamma_k}{\sqrt{\beta_k^2 + \kappa^2\gamma_k^2}}\right), \qquad \lim_{k\rightarrow \infty} \mu_k = \mu_{\infty} \coloneqq \frac{\kappa \gamma_{\infty}}{\sqrt{\beta_{\infty}^2 + \kappa^2\gamma_{\infty}^2}},\]
\noindent where $\gamma_k \coloneqq \max\{\tau_k,1\}$ and $\gamma_{\infty} = \max\{\tau_{\infty},1\}$ (noting that $\mu_{\infty} = 0,\ \textnormal{if }\beta_{\infty} = \infty$).
\end{theorem}
\begin{proof}
\noindent The proof is omitted since it follows by direct application of \cite[Theorem 2.5]{SIAMOpt:Lietal} (see also \cite[Theorem 2]{SIAMJCO:Rock}).
\end{proof}
\begin{remark}
Following \textnormal{\cite[Remarks 2, 3]{SIAMOpt:Lietal}}, we can choose a non-increasing sequence $\{\delta_k'\}_{k=0}^{\infty}$ and a large enough $\beta_0$ such that $\mu_0 < 1$, which in turn implies that $\mu_k \leq \mu_0 < 1$, yielding a global linear convergence of both $\textnormal{dist}\big((x_{k},y_k,z_k),T_{\ell}^{-1}(0)\big)$ as well as $\textnormal{dist}_{R_k}\big((x_{k},y_k,z_k),T_{\ell}^{-1}(0)\big)$, assuming that the starting point of the algorithm satisfies the assumption stated in Theorem \textnormal{\ref{thm:local convergence}}. On the other hand, as is implicitly mentioned in Theorem \textnormal{\ref{thm:local convergence}}, if $\beta_k$ is forced to increase indefinitely, we obtain a local superlinear convergence rate (notice that $\mu_{\infty} = 0$ if $\beta_{\infty} = \infty$).
\end{remark}

\section{Semismooth Newton method} \label{sec: SSN method}
\par In this section we briefly present a standard semismooth Newton (SSN) scheme suitable for the solution of problem \eqref{primal-dual PMM main sub-problem}, appearing in Algorithm \ref{primal-dual PMM algorithm}. 
More specifically, given the estimates $(x_k,y_k,z_k)$ as well as the penalty parameters $\beta_k,\ \rho_k$, we apply SSN to approximately solve \eqref{PD-SSN-PMM Proximal Optimality Conditions}. Given any $\zeta_k > 0$, the optimality conditions in \eqref{PD-SSN-PMM Proximal Optimality Conditions} can equivalently be written as
\begin{equation} \label{PD-SSN-PMM Proximal Optimality Conditions-SMOOTH}
\widehat{F}_{\beta_k,\rho_k,\zeta_k}\left(x,y\right) \coloneqq \begin{bmatrix} x  - \textbf{prox}_{\zeta_k g}\left(x - \zeta_k   r_{\beta_k,\rho_k}\left(x,y\right) \right) \\
\zeta_k\left(Ax + \beta_k^{-1} \left(y-y_k\right) -b\right)
\end{bmatrix}  = \begin{bmatrix}
0\\
0
\end{bmatrix},
\end{equation}
\noindent which follows from the properties of the $\textbf{prox}_{\zeta_k g}(\cdot)$ operator. We set $x_{k_0} = x_k$, $y_{k_0} = y_k$, and at every iteration $j$ of SSN, we approximately solve a system of the following form:
\begin{equation} \label{Primal-dual SSN system}
M_{k_j}\begin{bmatrix}
d_x\\
d_y
\end{bmatrix}= - \widehat{F}_{\beta_k,\rho_k,\zeta_k}\left(x_{k_j},y_{k_j}\right),
\end{equation}
\noindent where $M_{k_j} \in \mathscr{M}_{k_j}$, with
\begin{equation} \label{Clarke SSN matrices}
\begin{split}
\mathscr{M}_{k_j}   \coloneqq \Bigg\{   &\ M = \begin{bsmallmatrix} M_1 & M_2  \\ \zeta_kA & \zeta_k\beta_k^{-1}I_m \end{bsmallmatrix} \in \mathbb{R}^{(n+m)\times (n+m)} \colon  \widehat{u}_{k_j} = x_{k_j} - \zeta_k r_{\beta_k,\rho_k}\left(x_{k_j},y_{k_j}\right)
 ,\\
&\ H\left(x_{k_j},y_{k_j}\right) \in \partial_x^C\left(r_{\beta_k,\rho_k}\left(x_{k_j},y_{k_j}\right)\right),\quad \widehat{B}_{k_j}\left(\widehat{u}_{k_j}\right) \in \partial_x^C\left(\textbf{prox}_{\zeta_k g}\left(\widehat{u}_{k_j}\right)\right), \\
&\ M_1 = \left(I-\widehat{B}_{k_j}(\widehat{u}_{k_j})\right) + \zeta_k \widehat{B}_{k_j}\left(\widehat{u}_{k_j}\right)H\left(x_{k_j},y_{k_j}\right)\quad M_2 = -\zeta_k\widehat{B}_{k_j}\left(\widehat{u}_{k_j}\right)A^\top,\quad\Bigg\}.
\end{split}
\end{equation}
\noindent The symbol $\partial_x^C(\cdot)$ denotes the \emph{Clarke subdifferential} of a function (see \cite{JWS:Clarke}) with respect to $x$, which can be obtained as the convex hull of the \emph{Bouligand subdifferential} (\cite{JWS:Clarke}). Any element of the Clarke subdifferential is a \emph{Newton derivative} (see \cite[Chapter 13]{arXiv:ClasValk}), since $r_{\beta_k,\rho_k}(\cdot,y)$ and $g(\cdot)$ are \emph{piecewise continuously differentiable} and \emph{regular functions}. Using \cite[Theorem 14.7]{arXiv:ClasValk}, we obtain that for any $i \in \{1,\ldots,n\}$:
\begin{equation*}
\partial_{w_i}^C \left(\Pi_{[l_i,u_i]}(w_i)\right)= \begin{cases} 
\{1\}, &\qquad \textnormal{if}\quad w_i \in (l_i,u_i),\\
\{0\}, &\qquad \textnormal{if}\quad w_i \notin [l_i,u_i],\\
[0,1], &\qquad \textnormal{if}\quad w_i \in \{l_i,u_i\}.
\end{cases}
\end{equation*}
\noindent Furthermore, since $g(x) = \|Dx\|_1$, where $D$ positive semi-definite and diagonal, we have
\begin{equation*}
\left( \textbf{prox}_{\zeta_k g}\left(w\right)\right)_i = \max \Big\{ \left|w_i\right|- \zeta_k D_{(i,i)}, 0 \Big\} \textnormal{sign}(w_i),
\end{equation*}
\noindent where $\textnormal{sign}(\cdot)$ represents the sign of a scalar, and
\begin{equation*}
\left(\partial_{w}^C \left(\textbf{prox}_{\zeta_k g}\left(w\right)\right)\right)_i= \begin{cases} 
\{1\}, &\qquad \textnormal{if}\quad \left| w_i \right| > \zeta_k D_{(i,i)},\ \textnormal{or}\quad D_{(i,i)} = 0,\\
\{0\}, &\qquad \textnormal{if}\quad \left| w_i \right| < \zeta_k D_{(i,i)},\\
[0,1], &\qquad \textnormal{if}\quad \left| w_i \right| = \zeta_k D_{(i,i)}.
\end{cases}
\end{equation*}
\par For computational as well as theoretical reasons, we always choose matrices $M_{k_j}$ from the Bouligand subdifferential. The computational reasons for this choice will become apparent in the following subsection. On the other hand, it is well-known (see \cite[Theorem 4]{CAM:MartQi}) that an inexact semismooth Newton scheme using the Bouligand subdifferential converges at a local linear rate (assuming that the linear systems are solved up to an appropriate accuracy), if the equation in \eqref{PD-SSN-PMM Proximal Optimality Conditions-SMOOTH} is \emph{BD-regular} at the optimum $(x_k^*,y_k^*)$ (that is, each element of the Bouligand subdifferential of $\widehat{F}_{\beta_k,\rho_k,\zeta_k}(x_k^*,y_k^*)$ is nonsingular). Note, however, that since we employ the semismooth Newton scheme to solve the sub-problems arising from Algorithm \ref{primal-dual PMM algorithm}, we obtain that the resulting nonsmooth equations are indeed BD-regular for every outer iteration $k \geq 0$. Thus, assuming that the associated linear systems are solved up to a sufficient accuracy (see \cite[Theorem 4]{CAM:MartQi}), we obtain local linear convergence rate of the resulting inexact semismooth Newton scheme. If, additionally, the solution accuracy of the associated linear systems is increased at a suitable rate, the resulting local rate can be superlinear. 
\par We complete the derivation of the SSN by applying backtracking line-search on an appropriate primal-dual merit function. Then, under additional regularity assumptions one can show that SSN is globally convergent. To that end, we write the resulting primal-dual sub-problem as an $\ell_1$-regularized convex instance, by using a generalized primal-dual augmented Lagrangian merit function (e.g. see \cite{COAP:GillRobi}), i.e.
\[\hat{\psi}(x,y) = \hat{\phi}(x,y) + g(x), \qquad \hat{\phi}(x,y) \coloneqq \phi(x) + \frac{\beta_k}{2}\|Ax + \beta_k^{-1}(y-y_k) - b\|^2, \]
\noindent and the SSN sub-problem can be expressed as $\min_{x,y} \hat{\psi}(x,y)$. If $g(x) = 0$, then this smooth primal-dual merit function can be used to globalize the SSN, without any additional assumptions. For properties as well as an analysis of this merit function, we refer the reader to \cite{COAP:GillRobi}. In the nonsmooth case we have to resort to a different globalization strategy. Here we use the following merit function to globalize the SSN:
\begin{equation} \label{merit function for SSN globalization}
\Theta(x,y) \coloneqq  \left\|\widehat{F}_{\beta_k,\rho_k,\zeta_k}\left(x_{k_j},y_{k_j}\right)\right\|^{2}.
\end{equation}
\noindent This function is very often employed when globalizing SSN schemes applied to nonsmooth equations of the form of \eqref{PD-SSN-PMM Proximal Optimality Conditions-SMOOTH} (also known as the \emph{natural map}) by means of line-search. Indeed, its directional derivatives can be computed easily, assuming that the Bouligand subdifferential is exploited (see for example the analyses in \cite{MathOR:Han,InverseProbs:HansRaasch,CAM:MartQi} and the references therein).  Algorithm \ref{primal-dual SNM algorithm} outlines a semismooth Newton method for the approximate solution of \eqref{primal-dual PMM main sub-problem}. We assume that the associated linear systems are approximately solved by means of a Krylov subspace method. An analysis of the effect of errors arising from the use of Krylov methods within SSN applied to nonsmooth equations can be found in \cite{NLAA:ChenQi}. 
\renewcommand{\thealgorithm}{SSN}
\begin{algorithm}[!ht]
\caption{Semismooth Newton method}
    \label{primal-dual SNM algorithm}
    \textbf{Input:} $\epsilon_k > 0$, $\mu \in \left(0,\frac{1}{2}\right)$, $\delta \in (0,1)$, $\zeta_k > 0$, $\{\eta_j\}_{j = 0}^{\infty}$, $\eta_j \in (0,1)$, $x_{k_0} = x_k$,  $y_{k_0} = y_k$.
\begin{algorithmic}
\For {($j = 0,1,2,\ldots$)}
\State Choose $M_{k_j} \in \mathscr{M}_{k_j}$, where $\mathscr{M}_{k_j}$ is defined in \eqref{Clarke SSN matrices}, and solve
\[M_{k_j} \begin{bsmallmatrix}d_x\\d_y \end{bsmallmatrix} \approx - \widehat{F}_{\beta_k,\rho_k,\zeta_k}\left(x_{k_j},y_{k_j}\right),\]
\State such that $\left\| M_{k_j}\begin{bsmallmatrix}d_x\\d_y \end{bsmallmatrix} + \widehat{F}_{\beta_k,\rho_k,\zeta_k}\left(x_{k_j},y_{k_j}\right)\right\| \leq \eta_j \left\|\widehat{F}_{\beta_k,\rho_k,\zeta_k}\left(x_{k_j},y_{k_j}\right)\right\|.$ 
\State (Line-search) Set $\alpha_j = \delta^{m_j}$, where $m_j$ is the first non-negative integer for which:
\[\Theta\left(x_{k_j} + \delta^{m_j}d_x, y_{k_j} + \delta^{m_j} d_y\right) \leq \left(1 - 2 \mu \delta^{m_j}\right) \Theta\left(x_{k_j},y_{k_j}\right)\]
%\[\hat{\phi}\left(x_{k_j} + \delta^{m_j}d_x, y_{k_j} + \delta^{m_j}d_y\right)  \leq \hat{\phi}\left(x_{k_j},y_{k_j}\right)  + \mu \delta^{m_j}\left(\widehat{F}_{\beta_k,\rho_k,\zeta_k}\left(x_{k_j},y_{k_j}\right)\right)^\top d. \]
\State $x_{k_{j+1}} = x_{k_j} + \alpha_j d,\quad y_{k_{j+1}} = y_{k_j} + \alpha_j d_y$.
\If {$\left(\textnormal{dist}\left(0,F_{\beta_k,\rho_k}\left(x_{k_j},y_{k_j}\right)\right) \leq \epsilon_k\right)$}
\State \Return $(x_{k_{j+1}},y_{k_{j+1}})$.
\EndIf
\EndFor
\end{algorithmic}
\end{algorithm}
\par If $\eta_j$ is bounded above by an appropriately small number $\eta \in (0,1)$, then Algorithm \ref{primal-dual SNM algorithm} (assuming full-steps) is locally Q-linearly convergent (see \cite[Theorem 3]{CAM:MartQi}). Furthermore, if $\eta_j \rightarrow 0$, local superlinear convergence of Algorithm \ref{primal-dual SNM algorithm} follows directly from \cite[Theorem 4]{CAM:MartQi}. Similar analyses have been given in the literature in \cite{NLAA:ChenQi,InverseProbs:HansRaasch,MathOR:Qi}. Additionally, if the conditions outlined in \cite[(A1)--(A4)]{CAM:MartQi} hold (noting that in our case (A1) and (A3) hold automatically), then Algorithm \ref{primal-dual SNM algorithm} can be shown to be globally convergent. In particular, if we assume that the directional derivative of \eqref{merit function for SSN globalization} is continuous at the optimal point (see \cite[Equation (40)]{InverseProbs:HansRaasch}), we can mirror the analysis in \cite[Theorem 4.8]{InverseProbs:HansRaasch} to obtain global convergence of Algorithm \ref{primal-dual SNM algorithm}. This is omitted here, and the reader is referred to the analyses in \cite{InverseProbs:HansRaasch,CAM:MartQi} for additional details.
\par At this point we should mention certain alternatives to the merit function given in \eqref{merit function for SSN globalization}. There has been an extensive literature on the globalization of semismooth Newton methods for the solution of nonsmooth equations. Indeed, there have been developed approaches based on trust-region strategies (e.g. see \cite{SIAMOpt:Christofetal,MathProg:Dennis_etal,OptEng:MannelRund}), as well as line-search strategies based on smooth penalty functions (e.g. see the developments on the forward-backward envelope (FBE) \cite{IEEE_DC:Patrinos_etal,COAP:Stella_etal} or developments based on the proximal point method \cite{IEEE_CDC:Dhingra_etal}). In particular, line-search strategies based on the forward-backward envelope can be shown to yield globally convergent SSN schemes in our case without any additional assumptions. However, we have chosen to employ \eqref{merit function for SSN globalization} based on numerical considerations. Indeed, the natural map provides an active-set interpretation of Algorithm \ref{primal-dual SNM algorithm}, and contributes to its computational and memory efficiency. Additionally, \eqref{merit function for SSN globalization} is cheap to evaluate, and Algorithm \ref{primal-dual SNM algorithm} performs reliably well for all the problems studied in this paper.
\subsection{The SSN linear systems}
\par The major bottleneck of the previously presented inner-outer scheme is the approximate solution of the associated linear systems in \eqref{Primal-dual SSN system}. Since Algorithm \ref{primal-dual SNM algorithm} does not require an exact solution, we can utilize preconditioned Krylov subspace solvers for the efficient solution of such systems.
\par Let $k,\ j \geq 0$ be some arbitrary iterations of Algorithm \ref{primal-dual PMM algorithm}, and \ref{primal-dual SNM algorithm}, respectively. We notice that any element $B_{k_j} \in \partial_x^C\left(\Pi_{\mathcal{K}}\left(\beta_k^{-1}z_{k} + x_{k_j}\right)\right)$ yields a Newton derivative (see \cite[Theorem 14.8]{arXiv:ClasValk}). The same applies for any $\widehat{B}_{k_j} \in \partial_x^C \left( \textbf{prox}_{\zeta_k g}\left( \widehat{u}_{k_j} \right)\right)$, where $\widehat{u}_{k_j} = x_{k_j} - \zeta_k r_{\beta_k,\rho_k}(x_{k_j},y_{k_j})$. Thus, we choose $B_{k_j},\ \widehat{B}_{k_j}$ from the Bouligand subdifferential to improve computational efficiency. We set $B_{k_j},\ \widehat{B}_{k_j}$ as diagonal matrices with
\begin{equation}\label{eqn: Clarke subdifferential of projection choice}
\begin{split}
B_{k_j,(i,i)} \coloneqq & \begin{cases} 
1, &\quad \textnormal{if}\  \beta_k^{-1}z_{k,i} + x_{k_j,i} \in (a_{l_i},a_{u_i}),\\
0, &\quad  \textnormal{otherwise}, 
\end{cases},\\  \widehat{B}_{k_j,(i,i)} \coloneqq & \begin{cases} 
1, &\quad \textnormal{if}\  \left| \widehat{u}_{k_j}\right| > \zeta_k D_{(i,i)},\ \textnormal{or}\  D_{(i,i)} = 0,\\
0, &\quad  \textnormal{otherwise},
\end{cases} 
\end{split}
\end{equation}
\noindent for $i \in \{1,\ldots,n\}$, where $\widehat{u}_{k_j}$ is defined in \eqref{Clarke SSN matrices}. We can now explicitly write \eqref{Primal-dual SSN system}, for inner-outer iteration $k_j$, in the following saddle-point form
\begin{equation} \label{eqn: explicit SSN linear system}
\underbrace{\begin{bmatrix}-G_{k_j} &\zeta_k\widehat{B}_{k_j} A^\top\\
\zeta_k A & \zeta_k\beta_k^{-1} I_m
\end{bmatrix}}_{M_{k_j}}\begin{bmatrix}
d_x\\
d_y
\end{bmatrix} = \begin{bmatrix}
x_{k_j} - \textbf{prox}_{\zeta_k g}\left(\widehat{u}_{k_j}\right)\\
\zeta_k \left(b - Ax_{k_j} - \beta_k^{-1}\left(y_{k_j}-y_k\right) \right)
\end{bmatrix},
\end{equation}
\noindent where $G_{k_j} \coloneqq \left(I_n - \widehat{B}_{k_j}\right) + \zeta_k \widehat{B}_{k_j}H_{k_j},\  H_{k_j} \coloneqq Q + \left(\beta_k + \rho_k^{-1}\right)I_n - \beta_k B_{k_j}.$ Driven from \eqref{eqn: Clarke subdifferential of projection choice}, we define two index sets $\widehat{\mathcal{B}}_j \coloneqq \left \{ i \in \{1,\ldots,n\} \colon \widehat{B}_{k_j} = 1 \right\},\  \widehat{\mathcal{N}}_j \coloneqq \{1,\ldots,n\}\setminus \widehat{\mathcal{B}}_j.$ Observe that $ A \widehat{B}_{k_j}  =  \left[ A_{\widehat{\mathcal{B}}_j}\ 0 \right]\mathscr{P}^\top$, where $\mathscr{P}$ is an appropriate permutation matrix. Then, we can write \eqref{eqn: explicit SSN linear system} as
\begin{equation*} 
\underbrace{\begin{bmatrix}-\zeta_k H_{k_j,\left(\widehat{\mathcal{B}}_j,\widehat{\mathcal{B}}_j\right)} &- \zeta_k H_{k_j,\left(\widehat{\mathcal{B}}_j,\widehat{\mathcal{N}}_j\right)}& \zeta_kA_{\widehat{\mathcal{B}}_j}^\top\\
0 & -I_{\left| \widehat{\mathcal{N}}_j\right|} & 0\\
\zeta_k A_{\widehat{\mathcal{B}}_j} & \zeta_k A_{\widehat{\mathcal{N}}_j} & \zeta_k\beta_k^{-1} I_m
\end{bmatrix}}_{\left[\mathscr{P}^\top\ I_m \right] M_{k_j}\begin{bmatrix}
\mathscr{P}\\ I_m
\end{bmatrix}} \begin{bmatrix}
d_{x,\widehat{\mathcal{B}}_j}\\
d_{x,\widehat{\mathcal{N}}_j}\\
d_y
\end{bmatrix} = \left[\mathscr{P}^\top\ -I_m \right] \widehat{F}_{\beta_k,\rho_k,\zeta_k}\left(x_{k_j},y_{k_j}\right).
\end{equation*}
\noindent From the second block equation we obtain $d_{x,\widehat{\mathcal{N}}_j} = -\left(x_{k_j} - \textbf{prox}_{\zeta_k g}\left( \widehat{u}_{k_j}\right)\right)_{\widehat{\mathcal{N}}_j}.$ Thus, system \eqref{eqn: explicit SSN linear system} is reduced to the following symmetric saddle-point system
\begin{equation} \label{eqn: detailed explicit SSN linear system}
\underbrace{\begin{bmatrix}- H_{k_j,\left(\widehat{\mathcal{B}}_j,\widehat{\mathcal{B}}_j\right)} & A_{\widehat{\mathcal{B}}_j}^\top\\
 A_{\widehat{\mathcal{B}}_j}  & \beta_k^{-1} I_m
\end{bmatrix}}_{\widehat{M}_{k_j}} \begin{bmatrix}
d_{x,\widehat{\mathcal{B}}_j}\\
d_y
\end{bmatrix} = \begin{bmatrix}
\zeta_k^{-1}\left(x_{k_j} - \textbf{prox}_{\zeta_k g} \left(\widehat{u}_{k_j}\right)\right)_{\widehat{\mathcal{B}}_j} +  H_{k_j,\left(\widehat{\mathcal{B}}_j,\widehat{\mathcal{N}}_j\right)} d_{x,\widehat{\mathcal{N}}_j}\\
 \left(b - Ax_{k_j} - \beta_k^{-1}\left(y_{k_j}-y_k \right) - A_{\widehat{\mathcal{N}}_j} d_{x,\widehat{\mathcal{N}}_j}\right)
\end{bmatrix},
\end{equation}
\noindent the coefficient matrix of which is \emph{quasi-definite} (see \cite{SIAMOpt:Vander}), invertible, and its conditioning can be directly controlled by the regularization parameters $(\beta_k,\rho_k)$ of Algorithm \ref{primal-dual PMM algorithm}.
\subsection{Preconditioning and iterative solution of the linear systems}
\par Next, we would like to construct an effective preconditioner for $\widehat{M}_{k_j}$. To that end, we define
\begin{equation} \label{eqn: preconditioner for explicit SSN linear system}
\widetilde{M}_{k_{j}} \coloneqq \begin{bmatrix}
 \widetilde{H}_{k_{j},\left(\widehat{\mathcal{B}}_j,\widehat{\mathcal{B}}_j\right)} & 0 \\
0 &  \left(A_{\widehat{\mathcal{B}}_j} E_{k_{j}} A_{\widehat{\mathcal{B}}_j}^\top + \beta_k^{-1} I_m\right)
\end{bmatrix},
\end{equation}
\noindent with $\widetilde{H}_{k_{j},\left(\widehat{\mathcal{B}}_j,\widehat{\mathcal{B}}_j\right)} \coloneqq \textnormal{Diag}\left(H_{k_{j},\left(\widehat{\mathcal{B}}_j,\widehat{\mathcal{B}}_j\right)}\right)$ and $E_{k_{j}} \in \mathbb{R}^{\left|\widehat{\mathcal{B}}_j  \right| \times \left|\widehat{\mathcal{B}}_j  \right|}$ the diagonal matrix defined as
\begin{equation} \label{eqn: preconditioner dropping matrix E}
E_{k_{j},(i,i)} \coloneqq  \begin{cases} 
\widetilde{H}_{k_{j},\left(\widehat{\mathcal{B}}_{j,i},\widehat{\mathcal{B}}_{j,i}\right)}^{-1}, &\qquad \textnormal{if}\quad \beta_k^{-1}z_{k,\widehat{\mathcal{B}}_{j,i}} + x_{k_{j},\widehat{\mathcal{B}}_{j,i}} \in (l_{\widehat{\mathcal{B}}_{j,i}},u_{\widehat{\mathcal{B}}_{j,i}}),\\
0, &\qquad  \textnormal{otherwise},
\end{cases}
\end{equation}
\noindent where $\widehat{\mathcal{B}}_{j,i}$ denotes the $i$-th index of this index set, following the order imposed by $\mathscr{P}$.
\par The preconditioner in \eqref{eqn: preconditioner for explicit SSN linear system} is an extension of the preconditioner proposed in \cite{NLAA:BergGondMartPearPoug} for the solution of linear systems arising from the application of a regularized interior point method to convex quadratic programming. Being a diagonal matrix, $E_{k_{j}}$ yields a sparse approximation of the Schur complement of the saddle-point matrix in \eqref{eqn: preconditioner for explicit SSN linear system}. This approximation is then used to construct a positive definite block-diagonal preconditioner (i.e. $\widetilde{M}_{k_{j}}$), which can be used within a symmetric Krylov solver, like the minimum residual (MINRES) method (see \cite{PaigeSaundersSIAMNumAnal}). The $(2,2)$ block of $\widetilde{M}_{k_j}$ can be inverted via a Cholesky decomposition.
\par Next, we analyze the spectral properties of the preconditioned matrix $(\widetilde{M}_{k_{j}})^{-1}\widehat{M}_{k_j}$. Let 
\[\widehat{S}_{k_j} \coloneqq \left(A_{\widehat{\mathcal{B}}_j}  \widetilde{H}_{k_j,\left(\widehat{\mathcal{B}}_j,\widehat{\mathcal{B}}_j\right)}^{-1} A_{\widehat{\mathcal{B}}_j}^\top + \beta_k^{-1} I_m\right),\qquad \widetilde{S}_{k_j} \coloneqq \left(A_{\widehat{\mathcal{B}}_j} E_{k_{j}} A_{\widehat{\mathcal{B}}_j}^\top + \beta_k^{-1} I_m\right). \]
\noindent In the following lemma, we bound the eigenvalues of the preconditioned matrix $\widetilde{S}_{k_j}^{-1}\widehat{S}_{k_j}$. This is subsequently used to analyze the spectrum of $\widetilde{M}_{k_j}^{-1}\widehat{M}_{k_j}$.
\begin{lemma} \label{lemma: spectral properties of approximate preconditioned Schur complement}
For any iterates $k$ and $j$ of Algorithms \textnormal{\ref{primal-dual PMM algorithm}} and \textnormal{\ref{primal-dual SNM algorithm}}, respectively, we have
\[1 \leq \lambda \leq  1+\sigma_{\max}^2(A)\left(\frac{1}{1+\beta^{-2}_{\infty}\tau_{\infty}}\right),\]
\noindent where $\lambda \in \lambda\left( \widetilde{S}_{k_j}^{-1}\widehat{S}_{k_j}\right),$ and $\beta_{\infty},\ \tau_{\infty}$ are defined in Algorithm \textnormal{\ref{primal-dual PMM algorithm}}.
\end{lemma}
\begin{proof}
\noindent Consider the preconditioned matrix $\widetilde{S}_{k_j}^{-1}\widehat{S}_{k_j}$, and let $(\lambda,u)$ be its eigenpair. Then, $\lambda$ must satisfy the following equation:
\[\lambda = \frac{u^\top \left( A_{\mathcal{B}} D_{\mathcal{B}} A_{\mathcal{B}}^\top + \beta_k^{-1} I_m +  A_{\mathcal{N}} D_{\mathcal{N}} A_{\mathcal{N}}^\top\right) u}{u^\top \left( A_{\mathcal{B}} D_{\mathcal{B}} A_{\mathcal{B}}^\top + \beta_k^{-1 }I_m\right)u}, \]
\noindent where $\mathcal{B} = \left\{i \in \widehat{\mathcal{B}}_j \colon B_{k_j,(i,i)} = 1 \right\}$, $\mathcal{N} = \widehat{\mathcal{B}}_j \setminus \mathcal{B}$, $D_{\mathcal{B}} = \widetilde{H}_{k_j,\left(\mathcal{B},\mathcal{B}\right)}^{-1}$, $D_{\mathcal{N}} = \widetilde{H}_{k_j,\left(\mathcal{N},\mathcal{N}\right)}^{-1}$. The above equality holds since $\widetilde{H}_{k_j}$ is a diagonal matrix, and $E_{k_j,(i,i)} = 0$ for every $i$ such that $\widehat{\mathcal{B}}_{j,i} \notin \mathcal{B}$ (indeed, see the definition in \eqref{eqn: preconditioner dropping matrix E}). Hence, from positive semi-definiteness of $Q$, we obtain
\begin{equation*}
\begin{split}
 1 \leq \lambda = 1 + \frac{u^\top \left( A_{\mathcal{N}} D_{\mathcal{N}} A_{\mathcal{N}}\right) u}{u^\top \left( A_{\mathcal{B}} D_{\mathcal{B}} A_{\mathcal{B}}^\top + \beta_k^{-1 }I_m \right) u} &\ \leq 1 + \beta_k\sigma_{\max}^2(A_{\mathcal{N}})\left(\beta_k+\rho_k^{-1}\right)^{-1} \\
 &\ \leq 1 + \sigma_{\max}^2(A)\left(\frac{1}{1+\beta_k^{-2}\tau_k}\right) \leq  1+\sigma_{\max}^2(A)\left(\frac{1}{1+\beta_{\infty}^{-2}\tau_{\infty}}\right),
\end{split}
\end{equation*}
\noindent where we used that $\rho_k = \beta_k/\tau_k$, $\beta_k \leq \beta_{\infty}$, and $\tau_k \geq \tau_{\infty}$. 
\end{proof}
\par Given Lemma \ref{lemma: spectral properties of approximate preconditioned Schur complement}, we are now able to invoke \cite[Theorem 3]{NLAA:BergGondMartPearPoug} to characterize the spectral properties of the preconditioned matrix $\widetilde{M}_{k_j}^{-1}M_{k_j}$. Let 
\begin{equation*}
\begin{split}
\bar{S}_{k_j} \coloneqq & \left(\widetilde{S}_{k_j}\right)^{-\frac{1}{2}}\widehat{S}_{k_j}\left(\widetilde{S}_{k_j}\right)^{-\frac{1}{2}},\qquad 
\bar{H}_{k_j} \coloneqq \widetilde{H}_{k_j,\left(\widehat{\mathcal{B}}_j,\widehat{\mathcal{B}}_j\right)}^{-1/2} H_{k_j,\left(\widehat{\mathcal{B}}_j,\widehat{\mathcal{B}}_j\right)} \widetilde{H}_{k_j,\left(\widehat{\mathcal{B}}_j,\widehat{\mathcal{B}}_j\right)}^{-1/2},\\
\alpha_{NE} \coloneqq &\ \lambda_{\min}(\bar{S}_{k_j}),\quad  \beta_{NE} \coloneqq \lambda_{\max}(\bar{S}_{k_j}),\quad  \alpha_{H} \coloneqq  \lambda_{\min}(\bar{H}_{k_j}),\quad \beta_{H} \coloneqq \lambda_{\max}(\bar{H}_{k_j}).
\end{split}
\end{equation*}
\noindent Notice that Lemma \ref{lemma: spectral properties of approximate preconditioned Schur complement} yields upper and lower bounds for $\alpha_{NE}$ and $\beta_{NE}$. From the definition of $\bar{H}_{k_j}$ we can also obtain that $\alpha_{H} \leq 1 \leq \beta_{H}$ (see \cite{NLAA:BergGondMartPearPoug}). We are now ready to state the spectral properties of the preconditioned matrix  $\widetilde{M}_{k_j}^{-1}\widehat{M}_{k_j}$.
\begin{theorem} \label{thm: spectral properties of preconditioned matrix}
Let $k$ and $j$ be some arbitrary iterates of Algorithms \textnormal{\ref{primal-dual PMM algorithm} and \ref{primal-dual SNM algorithm}}, respectively. Then, the eigenvalues of $\widetilde{M}_{k_j}^{-1}\widehat{M}_{k_j}$ lie in the union of the following intervals:
\[I_{-} \coloneqq \left[-\beta_{H} -\sqrt{\beta_{NE}}, -\alpha_{H}\right],\qquad I_+ \coloneqq \left[\frac{1}{1+\beta_{H}},1+ \sqrt{\beta_{NE}-1}\right].\]
\end{theorem}
\begin{proof}
\noindent The proof follows by direct application of \cite[Theorem 3]{NLAA:BergGondMartPearPoug}.
\end{proof}
\begin{remark}
By combining Lemma \textnormal{\ref{lemma: spectral properties of approximate preconditioned Schur complement}} with Theorem \textnormal{\ref{thm: spectral properties of preconditioned matrix}}, we can observe that the eigenvalues of the preconditioned matrix $\widetilde{M}_{k_j}^{-1} \widehat{M}_{k_j}$ are not deteriorating as $\beta_k \rightarrow \infty$. In other words, the preconditioner is robust with respect to the penalty parameters $\beta_k,\ \rho_k$ of Algorithm \textnormal{\ref{primal-dual PMM algorithm}}. Furthermore, our choices of $B_{k_j},\ \widehat{B}_{k_j}$ in \eqref{eqn: Clarke subdifferential of projection choice} serve the purpose of further sparsifying the preconditioner in \eqref{eqn: preconditioner for explicit SSN linear system}, thus potentially further sparsifying its Cholesky decomposition. We note that the preconditioner in \textnormal{\eqref{eqn: preconditioner for explicit SSN linear system}} is an efficient choice if $n \geq m$, which is the case in the experiments considered in Section \textnormal{\ref{sec: applications}}. However, any of the preconditioners given in \textnormal{\cite{arXiv:GondPougkPears}} can be directly applied for systems appearing in the proposed inner-outer scheme. For example, in the case where $m \geq n$ one could employ the preconditioner given in \textnormal{\cite[Section 3.2]{arXiv:GondPougkPears}}, which would be a more efficient choice as compared to that given in \textnormal{\eqref{eqn: preconditioner for explicit SSN linear system}}. This is omitted here for ease of presentation.
\end{remark}

\begin{remark}
For problems solved within this work, a diagonal approximation of the Hessian (within the preconditioner) seems sufficient to deliver very good performance. Indeed, this is the case for a wide range of problems. However, in certain instances, one might consider non-diagonal approximations of the Hessian. In that case, the preconditioner in \textnormal{\eqref{eqn: preconditioner for explicit SSN linear system}} can be readily generalized and analyzed, following the developments in \textnormal{\cite[Section 3.1]{arXiv:GondPougkPears}}. 
\end{remark}
\section{Warm-starting} \label{sec: warm start}
\par Next, we would like to find a starting point $(x_0,y_0,z_0)$ for Algorithm \ref{primal-dual PMM algorithm} that is relatively close to the solution of \eqref{primal problem}--\eqref{dual problem}, since then we can expect to observe early linear convergence of Algorithm \ref{primal-dual PMM algorithm}, and obtain active-sets that contain only a small number of variables (thus reducing the memory and CPU requirements of the method). To that end, we employ a proximal alternating direction method of multipliers (pADMM; e.g. see \cite{SciComp:DengYin}) to find an approximate solution of \eqref{primal problem}--\eqref{dual problem}. We reformulate \eqref{primal problem} as follows:
\begin{equation} \label{primal problem ADMM reformulation} \tag{P'}
\underset{x \in \mathbb{R}^n, w \in \mathbb{R}^n}{\text{min}} \  c^\top x + \frac{1}{2} x^\top Q x + g(w) + \delta_{\mathcal{K}}(w), \qquad \text{s.t.}  \  Ax = b, \quad w - x = 0.
\end{equation}
\par Given a penalty $\sigma > 0$, we associate the following augmented Lagrangian to \eqref{primal problem ADMM reformulation}
\begin{equation*}
\begin{split}
\scalemath{1}{\widehat{\mathcal{L}}_{\sigma}(x,w,y_1, y_2) \coloneqq  c^\top x + \frac{1}{2} x^\top Q x + g(w) + \delta_{\mathcal{K}}(w) - \begin{bmatrix} y_1 \\ y_2 \end{bmatrix}^\top \begin{bmatrix} Ax - b\\ w - x \end{bmatrix} + \frac{\sigma}{2}\|Ax-b\|^2 + \frac{\sigma}{2}\|w-x\|^2. }
\end{split}
\end{equation*}
\noindent Algorithm \ref{proximal ADMM algorithm} summarizes a proximal ADMM for the approximate solution of \eqref{primal problem ADMM reformulation}.
\renewcommand{\thealgorithm}{pADMM}
\begin{algorithm}[!ht]
\caption{proximal ADMM}
    \label{proximal ADMM algorithm}
    \textbf{Input:}  $\sigma > 0$, $R_x \succ 0$, $\gamma \in \left(0,\frac{1+\sqrt{5}}{2}\right)$, $(x_0,w_0,y_{1,0},y_{2,0}) \in \mathbb{R}^{3 n + m}$.
\begin{algorithmic}
\For {($k = 0,1,2,\ldots$)}\begin{align*} w_{k+1} &=   \underset{w}{\arg\min}\left\{\widehat{\mathcal{L}}_{\sigma}\left(x_k,w,y_{1,k},y_{2,k}\right) \right\} \equiv \Pi_{\mathcal{K}}\left(\textbf{prox}_{\sigma^{-1} g}\left(x_k + \sigma^{-1} y_{2,k}\right)\right).\\
 x_{k+1} &= \underset{x}{\arg\min}\left\{\widehat{\mathcal{L}}_{\sigma}\left(x,w_{k+1},y_{1,k},y_{2,k}\right) + \frac{1}{2}\|x-x_k\|_{R_x}^2\right\}.\\
 \begin{bmatrix} y_{1,k+1}\\ y_{2,k+1} \end{bmatrix} &= \begin{bmatrix} y_{1,k}\\y_{2,k} \end{bmatrix} - \gamma \sigma \begin{bmatrix} Ax_{k+1}-b\\ w_{k+1}-x_{k+1} \end{bmatrix}
 \end{align*}
\EndFor
\end{algorithmic}
\end{algorithm}
\par A detailed convergence analysis of Algorithm \ref{proximal ADMM algorithm} can be found in \cite{SciComp:DengYin}. We choose $R_x \succ 0$ as a means of reducing the memory requirements of this approach. More specifically, given some constant $\hat{\sigma} > 0$, such that $\hat{\sigma}I_n - \textnormal{Off}(Q) \succ 0$, we define $R_x = \hat{\sigma}I_n - \textnormal{Off}(Q)$. 
\par The first and third steps of Algorithm \ref{proximal ADMM algorithm} are trivial, and the main computational bottleneck lies in the second sub-problem. Merging it with the subsequent dual updates yields:
\begin{equation*}
\begin{bmatrix}
-\gamma\left(\textnormal{Diag}(Q)+\hat{\sigma}I_n\right) & A^\top & -I_n\\
A & \frac{1}{\gamma \sigma}I_m & 0\\
-I_n & 0 & \frac{1}{\gamma\sigma} I_n
\end{bmatrix}
\begin{bmatrix}
x\\
y_1\\
y_2
\end{bmatrix}
= \begin{bmatrix}
\gamma\left(c - R_{x} x_k \right) + (1-\gamma)\left(A^\top y_{1,k} - y_{2,k}\right)\\
b + \frac{1}{\gamma \sigma} y_{1,k}\\
 \frac{1}{\gamma\sigma} y_{2,k} -w_{k+1} 
\end{bmatrix}.
\end{equation*}
\noindent Assuming we have sufficient memory, the previous system can be solved by means of a $LDL^\top$ factorization, since the coefficient matrix is symmetric quasi-definite. The benefit of this approach is that a single factorization can be utilized for all iterations of Algorithm \ref{proximal ADMM algorithm}. If the available memory is not sufficient, or the problem under consideration is structured (e.g. its data matrices belong to an appropriate structured matrix sequence), one might attempt to solve the previous system using a symmetric solver like MINRES (\cite{PaigeSaundersSIAMNumAnal}) or CG (\cite{HestenesSteifelPCG}). In this case, a preconditioner for either MINRES or CG would have to be computed only once (e.g. see the solver in \cite{SIAMX:Pougketal}). Alternatively, a different choice for $R_x$ could further reduce the memory requirements of this warm-starting scheme, even to the point of making it \emph{matrix-free} (e.g. see the prox-linear ADMM given in \cite[Section 1.1]{SciComp:DengYin}).
\par Finally, once an approximate solution $(\tilde{x},\tilde{w},\tilde{y}_{1},\tilde{y}_2)$ is retrieved, we set the starting point of Algorithm \ref{primal-dual PMM algorithm} as $(x_0,y_0,z_0) = (\tilde{x},\tilde{y}_1,z)$, where $z = \tilde{y}_2 - \Pi_{\partial g(\tilde{w})}\left(\tilde{y}_2\right)$. An optimal primal-dual solution of \eqref{primal problem ADMM reformulation} is such that $\tilde{y}^*_2 \in \partial g\left(\tilde{w}^*\right) + \partial \delta_{\mathcal{K}}\left(\tilde{w}^*\right)$, so the characterization of $z$ in Algorithm \ref{primal-dual PMM algorithm} can be obtained as shown in Appendix \ref{Appendix: termination criteria}.
\section{Applications and numerical results} \label{sec: applications}
\par In what follows we showcase the effectiveness of the proposed approach on  elastic-net linear regression and $L^1$-regularized optimization problems with partial differential equation constraints. For each problem, we test the active-set method against the IP-PMM (a MATLAB-based regularized interior point method) presented in \cite{NLAA:BergGondMartPearPoug} (the code of which can be found on GitHub\footnote{\url{https://github.com/spougkakiotis/IP-PMM_QP_Solver}}). We choose to compare the proposed approach against this method for multiple reasons. Firstly, both methods are written in MATLAB. Additionally, the regularized interior point method given in \cite{NLAA:BergGondMartPearPoug} has been demonstrated to be very robust and efficient in a plethora of optimization problems (see also \cite{arXiv:GondPougkPears}), and has been compared against various other optimization schemes (see \cite{SIREV:DeSimone_etal}). Finally, both methods employ preconditioned Krylov subspace solvers, with the same preconditioning approach, for the solution of their respective linear systems. We also compare the two second-order solvers against the well-known OSQP solver (see \cite{osqp}; the code can be found on Github\footnote{\url{https://github.com/osqp/osqp-matlab}}), which is an ADMM (C-based) method often employed in the literature for problems of the form of \eqref{primal problem}, using appropriate quadratic programming reformulations. In order to make the comparison fair, in all the experiments to follow we alter the termination criteria of IP-PMM to match those given in Appendix \ref{Appendix: termination criteria}. OSQP uses more relaxed termination criteria, but as can be seen in the results to follow, it does not manage to compete reliably with the second-order solvers. A follow-up numerical study on various additional applications has been compiled in an accompanying paper.
\par The active-set (AS) solver is written in MATLAB (the code is available on GitHub\footnote{\url{https://github.com/spougkakiotis/SSN_PMM}}). The experiments were run on MATLAB 2019a, on a PC with a 2.2GHz Intel core i7-8750H processor (hexa-core), 16GB RAM, using the Windows 10 operating system. The warm-starting mechanism proposed in Section \ref{sec: warm start} is allowed to run for at most 100 iterations, and is terminated if it reaches a 3-digit accurate solution. Its associated linear systems are solved using a single call to the $\texttt{ldl}$ decomposition of MATLAB. In order to accelerate the convergence of the SSN solver, we employ the following \emph{predictor-corrector-like} heuristic: at the first SSN iteration we accept a full step, without employing line-search. Then, line-search is reactivated for every subsequent iteration. We observe that this heuristic improves the practical performance of the proposed method, allowing for a rapid convergence of the SSN scheme in cases where the current estimate is close to the optimal solution of the associated sub-problem, and not hindering its robustness. Any linear system solved within SSN-PMM is solved using preconditioned MINRES, and the (2,2) block of the preconditioner (given in \eqref{eqn: preconditioner for explicit SSN linear system}) is inverted using MATLAB's \texttt{chol} function. The penalty parameters of PMM are tuned as follows: we initially set $\beta = 10^2$ and $\rho = 5\cdot 10^2$. At the end of each call to SSN, we increase them at a suitable rate. In particular, these are increased more rapidly if the dual or primal infeasibilities, respectively, have sufficiently decreased. If not, we increase the penalties more conservatively. The termination criteria of the implemented approach are given in Appendix \ref{Appendix: termination criteria}. All other implementation details follow exactly the developments in Sections \ref{sec: PAL penalties}, \ref{sec: SSN method}. 
\subsection{Elastic-net linear regression}
\par We first consider standard linear regression models
\[y_i = \beta_0 + \xi_i^\top \beta + \epsilon_i, \qquad i \in \{1,\ldots,N\}\]
\noindent where $\xi_i$ is a $d$-dimensional vector of covariates, $(\beta_0,\beta)$ are the regression coefficients and $\epsilon_i$ is some random error. Given some regularization parameters $\lambda > 0$, $\tau \in [0,1]$ for the elastic-net penalty (see \cite{CC01a,JRSS:Zou}), we would like to solve
\begin{equation} \label{eqn: linear regression model}
    \min_{(\beta_0,\beta)\ \in\ \mathbb{R} \times \mathbb{R}^d} \left\{\frac{1}{N} \sum_{i=1}^N \left(y_i - \beta_0 - \xi_i^\top \beta\right)^2 + \lambda \left(\tau \|\beta\|_1 + \frac{1-\tau}{2}\|\beta\|_2^2\right) \right\}.
\end{equation}
\paragraph{Real-world datasets:} We solve several instances of problem  \eqref{eqn: linear regression model}, arising from 4 regression datasets taken from the LIBSVM library (see \cite{CC01a}). Each dataset contains data points and features originating from some real-world application. Additional information on the datasets is collected in Table \ref{Table: Lasso regression datasets}, where we introduce name abbreviations for space considerations. 
\begin{table}[!ht] 
\centering
\caption{Lasso regression datasets.\label{Table: Lasso regression datasets}}
\scalebox{1}{
\begin{tabular}{lll}     
\specialrule{.2em}{.05em}{.05em} 
    \textbf{Name} & \textbf{\# of data points} & \textbf{\# of features} \\  \specialrule{.1em}{.3em}{.3em} 
abl $\equiv$ abalone & 4177 & 8 \\ 
cps $\equiv$ cpusmall & 8192 & 12 \\ 
%cdt $\equiv$ cadata & 20,640 & 8 \\ 
ypi $\equiv$ YearPredictionMSD & 463,715 & 90\\
e2t $\equiv$ E2006-tfidf & 16,087 &  150,360 \\ 
\specialrule{.2em}{.05em}{.05em} 
\end{tabular}}
\end{table}
\par We fix the tolerance as $\texttt{tol} = 10^{-4}$, and run all three methods (active-set (AS), IP-PMM and OSQP) on each of the 4 datasets for varying regularization values $(\lambda,\tau)$. We note that if a dataset is not scaled, then an absolute tolerance is requested from all solvers. We report the values of the regularization parameters $(\lambda,\tau)$, as well as the number of iterations performed by each method. For the active-set scheme we report the number of PMM, SSN and Krylov (in particular MINRES) iterations, as well as the total number of factorizations (denoted as Fact.) of the associated preconditioners. It is often the case that the preconditioner used in a previous SSN iteration does not have to be altered in a subsequent one, which in turn reduces the number of factorizations performed by the algorithm. For IP-PMM, we report the number of outer interior point and inner Krylov iterations, and for OSQP the number ADMM iterations. For each method, we also report the total CPU time required for convergence (or unsuccessful termination) in seconds, including the time needed to transform the problem to the format accepted by each solver. In all the numerical results that follow the lowest running time exhibited by a solver, assuming it successfully converged, is presented in bold. The results are collected in Table \ref{Table lasso regression: varying regularization}. 
\begin{table}[!ht]
\caption{Elacstic-net regression: varying regularization (absolute \texttt{tol} = $10^{-4}$). \label{Table lasso regression: varying regularization}}
\centering
\scalebox{0.84}{\begin{threeparttable} 
\begin{tabular}{llllllll}     
\specialrule{.2em}{.05em}{.05em} 
    \multirow{2}{*}{\textbf{Prob.}} & \multirow{2}{*}{$(\bm{\lambda},\bm{\tau})$} & \multicolumn{3}{c}{\textbf{Iterations}}     & \multicolumn{3}{c}{\textbf{Time (s)}}  \\   \cmidrule(l{2pt}r{2pt}){3-5} \cmidrule(l{2pt}r{2pt}){6-8}
&   &  {{PMM}(SSN)[Fact.]\{Krylov\}} & {{IP--PMM(Krylov)}}  & {{OSQP}}&  {{AS}} & {{IP--PMM}}  & {OSQP}    \\ \specialrule{.1em}{.3em}{.3em} 
 \multirow{6}{*}{{abl}} & ($10^{-2}$,0.8) &  5(7)[5]\{77\} &  12(491) & 325 & 0.03 & 0.06  & \textbf{0.01} \\ 
  & ($10^{-2}$,0.2) &  4(5)[4]\{55\} & 13(510)  & 250& 0.02 & 0.07 & \textbf{0.01}  \\ 
 & ($10^{-3}$,0.8) &  5(11)[6]\{117\} & 13(532) &  275& 0.03 & 0.06 & \textbf{0.01} \\
   & ($10^{-3}$,0.2) &  5(12)[6]\{130\} & 19(853) &250  & 0.04 & 0.10 &  \textbf{0.01} \\ 
        & ($10^{-4}$,0.8) &  6(7)[6]\{49\} & 15(644) &  375& \textbf{0.02} &  0.09& \textbf{0.02} \\
   & ($10^{-4}$,0.2) &  6(6)[6]\{42\} & 11(426) &450  & 0.03 & 0.05 &  \textbf{0.01} \\ \specialrule{.00002em}{.1em}{.1em}
    \multirow{6}{*}{{cps}}  & ($10^{-2}$,0.8) &  6(9)[11]\{286\} &  15(1051) & 525 & 0.05 & 0.09  &\textbf{0.01} \\ 
  & ($10^{-2}$,0.2) &  5(7)[5]\{113\} & 14(888)  &  400 & 0.03 & 0.09  & \textbf{0.01} \\ 
 & ($10^{-3}$,0.8) &  7(21)[9]\{344\} & 14(952) & 550 & 0.06 & 0.09 & \textbf{0.01}\\
   & ($10^{-3}$,0.2) &  6(9)[6]\{151\} & 14(947) & 675 & 0.04 & 0.09 &  \textbf{0.01}\\
   & ($10^{-4}$,0.8) &  7(10)[7]\{169\} & 14(961) & 750 & 0.04 & 0.09 & \textbf{0.01}\\
   & ($10^{-4}$,0.2) &  7(10)[7]\{169\} & 13(887) & 725 & 0.03 & 0.09 &  \textbf{0.01}\\ \specialrule{.00002em}{.1em}{.1em}
% \multirow{6}{*}{{cdt}} & ($10^{-2}$,0.8) &  17(46)[17]\{711\} &  $200(4786)^{\ddagger}$\tnote{1} & $4000^{\ddagger}$ & \textbf{0.11} & $0.70^{\ddagger}$ & $0.01^{\ddagger}$ \\ 
%  & ($10^{-2}$,0.2) &  17(46)[17]\{516\} & $200(5026)^{\ddagger}$  & $4000^{\ddagger}$& \textbf{0.09} &  $0.74^{\ddagger}$ & $0.01^{\ddagger}$\\ 
% & ($10^{-3}$,0.8) &  17(46)[17]\{521\} & $200(6099)^{\ddagger}$ & $4000^{\ddagger}$ & \textbf{0.10} & $0.80^{\ddagger}$ & $0.01^{\ddagger}$\\
%   & ($10^{-3}$,0.2) &  17(46)[17]\{520\} & $200(6554)^{\ddagger}$ & $4000^{\ddagger}$ & \textbf{0.10} & $0.83^{\ddagger}$ & $0.01^{\ddagger}$ \\
 %  & ($10^{-4}$,0.8) &  17(46)[17]\{521\} & 16(556) & $4000^{\ddagger}$ & 0.10 & \textbf{0.08} & $0.01^{\ddagger}$\\
 %  & ($10^{-4}$,0.2) &  17(46)[17]\{520\} & 16(556) & $4000^{\ddagger}$ & 0.10 & \textbf{0.08} & $0.01^{\ddagger}$ \\ \specialrule{.00002em}{.1em}{.1em}
\multirow{6}{*}{{ypi}} & ($10^{-2}$,0.8) &  6(36)[32]\{2494\} &  $200(9370)^{\ddagger}$ & $4000^{\ddagger}$ & \textbf{5.53} & $6.50^{\ddagger}$ & $5.40^{\ddagger}$\\ 
  & ($10^{-2}$,0.2) &  5(29)[26]\{2090\} & $200(10,282)^{\ddagger}$  & $4000^{\ddagger}$ & \textbf{5.53} &  $7.03^{\ddagger}$ &  $5.71^{\ddagger}$\\ 
 & ($10^{-3}$,0.8) &  5(32)[30]\{2283\} &  $200(12,009)^{\ddagger}$ & $4000^{\ddagger}$ & \textbf{5.57} & $6.60^{\ddagger}$ & $5.31^{\ddagger}$\\
   & ($10^{-3}$,0.2) &  6(24)[20]\{1928\} &  $200(17,980)^{\ddagger}$ &  $4000^{\ddagger}$ & \textbf{5.47} & $7.05^{\ddagger}$ &  $5.32^{\ddagger}$\\
   & ($10^{-4}$,0.8) &  6(19)[15]\{1607\} &  14(1091) & $4000^{\ddagger}$ & 5.50 & $\textbf{5.36}$ & $5.40^{\ddagger}$\\
   & ($10^{-4}$,0.2) &  6(15)[7]\{1334\} &  14(1129) &  $4000^{\ddagger}$ & \textbf{5.35} & $5.40$ &  $5.35^{\ddagger}$\\ \specialrule{.00002em}{.1em}{.1em}
   \multirow{6}{*}{{e2t}}& ($10^{-2}$,0.8) &  2(2)[2]\{4\} & $\dagger$\tnote{2}  & $\dagger$ & \textbf{878.43} & $\dagger$ & $\dagger$\\ 
  & ($10^{-2}$,0.2) &  2(2)[2]\{4\} &  $\dagger$ &  $\dagger$& \textbf{1165.83} & $\dagger$ & $\dagger$ \\ 
 & ($10^{-3}$,0.8) &  2(8)[2]\{28\} &$\dagger$ & $\dagger$ & \textbf{1442.56} & $\dagger$&$\dagger$ \\
   & ($10^{-3}$,0.2) &  2(8)[2]\{25\} & $\dagger$& $\dagger$ & \textbf{1362.64} & $\dagger$& $\dagger$ \\ 
    & ($10^{-4}$,0.8) &  2(8)[2]\{36\} &$\dagger$ & $\dagger$ & \textbf{1894.28} & $\dagger$&$\dagger$ \\
   & ($10^{-4}$,0.2) &  2(8)[2]\{31\} & $\dagger$& $\dagger$ & \textbf{1522.76} & $\dagger$& $\dagger$ \\ 
\specialrule{.2em}{.05em}{.05em} 
\end{tabular}
\begin{tablenotes}
\item[1] $\ddagger$ indicates that the solver reached the maximum number of iterations.
\item[2] $\dagger$ indicates that the solver ran out of memory.
\end{tablenotes}
\end{threeparttable}}
\end{table}
\par The results in Table \ref{Table lasso regression: varying regularization} indicate that the active-set scheme is more reliable and scalable than either IP-PMM or OSQP, while all three methods are competitive on the smaller well-conditioned instances. Interestingly enough, there are several cases where the active-set method requires fewer SSN iterations compared to the interior point iterations required by IP-PMM, thus performing significantly less Krylov iterations. A factor contributing to this behaviour is the warm-starting mechanism of the active-set scheme. IP-PMM employs the standard Mehrotra warm-starting scheme, \cite{SIAMOpt:Mehrotra}. Developing a more elaborate warm-starting scheme for IPMs requires significantly more effort, compared to the warm-starting mechanism given in Algorithm \ref{proximal ADMM algorithm}, due to the nature of the logarithmic barriers. IP-PMM and OSQP fail to converge for several ill-conditioned instances coming from the YearPredictionMSD dataset, while both IP-PMM and OSQP fail, due to memory requirements, when applied to instances coming from the largest dataset, namely E2006-tfidf. As expected, IP-PMM is more robust than OSQP. Finally, let us notice that both the active-set method and IP-PMM utilize the same preconditioning strategy and similar iterative linear algebra. However, the former solves a smaller problem formulation, while further reduces its memory requirements by only working with an active-set at each SSN iteration. This is reflected in the fact that the active-set scheme was able to reliably solve instances originating from E2006-tfidf. 
\subsection{PDE-constrained optimization} 
\par Next, we test the proposed methodology on some optimization problems with partial differential equation constraints. We consider optimal control problems of the following form:
\begin{equation} \label{generic inverse problem}
\begin{split}
\min_{\mathrm{y},\mathrm{u}} \     &\ \frac{1}{2}\| \rm{y} - \bar{\rm{y}}\|_{L^2(\Omega)}^2 + \frac{\alpha_1}{2}\|\rm{u}\|_{L^1(\Omega)}^2 + \frac{\alpha_2}{2}\|\rm{u}\|_{L^2(\Omega)}^2, \\
\text{s.t.}\ &\ \mathrm{D} \mathrm{y}(\bm{x}) + \mathrm{u}(\bm{x}) = \mathrm{g}(\bm{x}),\quad \mathrm{u_{a}}(\bm{x}) \leq \mathrm{u}(\bm{x})  \leq \mathrm{u_{b}}(\bm{x}),
\end{split}
\end{equation}
\noindent where $(\rm{y},\rm{u}) \in \text{H}^1(\Omega) \times \text{L}^2(\Omega)$, $\mathrm{D}$ is some linear differential operator, $\bm{x}$ is a $2$-dimensional spatial variable, and $\alpha_1,\ \alpha_2 \geq 0$ are the regularization parameters of the control variable. The problem is considered on a given compact spatial domain $\Omega$, where $\Omega \subset \mathbb{R}^{2}$ has boundary $\partial \Omega$, and is equipped with Dirichlet boundary conditions. The algebraic inequality constraints are assumed to hold a.e. on $\Omega$, while ${\rm u_a}$ and ${\rm u_b}$ may take the form of constants or functions of the spatial variables. 
\par We solve problem \eqref{generic inverse problem} via a discretize-then-optimize strategy. We employ the Q1 finite element discretization implemented in IFISS\footnote{\url{https://personalpages.manchester.ac.uk/staff/david.silvester/ifiss/default.htm}} (see \cite{IFISSACM,IFISSSIAMREVIEW}) which yields a sequence of $\ell_1$-regularized convex quadratic programming problems in the form of \eqref{primal problem}. We note that the discretization of the smooth parts of problem \eqref{generic inverse problem} follows a standarad Galerkin approach (e.g. see \cite{AMS:Trolzsch}), while the $L^1$ term is discretized by the \emph{nodal quadrature rule} as in \cite{COAP:SongChenYu} (an approximation that achieves a first-order convergence--see \cite{ESAIM:GerdDaniel}). In what follows, we consider Poisson's as well as the convection--diffusion state equations. 

\subsubsection{Poisson optimal control}
\par We first consider two-dimensional $L^1/L^2$-regularized Poisson optimal control problems. The problem is posed on $\Omega = (0,1)^2$. Following \cite[Section 5.1]{NLAA:PearsonPorcStoll}, we set constant control bounds $\mathrm{u}_a = -2$, $\mathrm{u}_b = 1.5$, and the desired state as $\bar{\mathrm{y}} = \sin(\pi x_1)\sin(\pi x_2)$. In Table \ref{Table Poisson optimal control: varying grid-size and L1 regularization}, we fix the $L^2$ regularization parameter to the value $\alpha_2 = 10^{-2}$, the (relative) tolerance to $\texttt{tol} = 10^{-6}$, and present the runs of the three methods (i.e. active-set, IP-PMM, OSQP) for varying $L^1$ regularization (i.e. $\alpha_1$) as well as grid size. We report the size of the resulting discretized problems (before any reformulation), the value of $\alpha_1$, the number iterations required by each solver, as well as the total time to convergence. 

\begin{table}[!ht]
\caption{Poisson control: varying grid size and $L^1$ regularization. $(\texttt{tol},\alpha_2) = (10^{-6},10^{-2})$.\label{Table Poisson optimal control: varying grid-size and L1 regularization}}
\centering
\scalebox{0.84}{\begin{threeparttable}
\begin{tabular}{llllllll}     
\specialrule{.2em}{.05em}{.05em} 
    \multirow{2}{*}{\textbf{$\bm{n}$}} & \multirow{2}{*}{$\bm{\alpha_1}$} & \multicolumn{3}{c}{\textbf{Iterations}}     &  \multicolumn{3}{c}{\textbf{Time (s)}}  \\   \cmidrule(l{2pt}r{2pt}){3-5} \cmidrule(l{2pt}r{2pt}){6-8}
&   &  {{PMM(SSN)[Fact.]\{Krylov\}}} & {{IP-PMM(Krylov)}}  & {{OSQP}} & AS & IP-PMM & OSQP      \\ \specialrule{.1em}{.3em}{.3em} 
 \multirow{4}{*}{$8.45\cdot 10^3$} & $10^{-2}$ &  13(58)[26]\{1278\} & 13(1353)  & 2075 &  \textbf{5.92} &  6.77 & 18.58  \\
  &   $10^{-4}$ &13(47)[30]\{572\} &14(1586) &  2150 &   \textbf{3.27} & 7.70 & 18.50\\
  &   $10^{-6}$ &  13(47)[30]\{574\} & 14(1586)  & 2225 &  \textbf{3.20} & 7.51 & 18.82\\ 
    &   $0$ & 13(47)[30]\{574\} &  14(1586)  &  2375 &  \textbf{3.38}  & 7.69 & 21.14\\ \specialrule{.00002em}{.1em}{.1em}
 \multirow{4}{*}{$3.32\cdot 10^4$} & $10^{-2}$ & 12(26)[22]\{262\} & 15(1514) &  $4000^{\ddagger}$\tnote{1} & \textbf{7.43}  & 24.85 & $143.40^{\ddagger}$  \\
  &   $10^{-4}$ & 12(26)[20]\{219\} & 15(1759) & $4000^{\ddagger}$ &  \textbf{6.75} & 27.93 & $149.42^{\ddagger}$ \\
  &   $10^{-6}$ & 12(27)[21]\{229\} & 15(1760) & $4000^{\ddagger}$ &  \textbf{6.92} & 27.87 & $145.92^{\ddagger}$ \\ 
    &   $0$ & 12(27)[21]\{229\} & 15(1760)  & $4000^{\ddagger}$ &  \textbf{6.92} & 29.09 & $166.84^{\ddagger}$\\ \specialrule{.00002em}{.1em}{.1em}
\multirow{4}{*}{$1.32\cdot 10^5$} & $10^{-2}$ & 13(13)[13]\{132\} & 13(981)  & $4000^{\ddagger}$ &  \textbf{25.16} & 81.17 & $843.76^{\ddagger}$\\
  &   $10^{-4}$ & 13(13)[13]\{129\} & 14(1269) &  $4000^{\ddagger}$ & \textbf{24.72} & 102.82 & $835.53^{\ddagger}$\\
  &   $10^{-6}$ & 13(13)[13]\{129\} & 14(1271) & $4000^{\ddagger}$ &  \textbf{24.89} & 103.68 & $823.21^{\ddagger}$ \\ 
    &   $0$ & 13(13)[13]\{129\} & 14(1271) & $4000^{\ddagger}$ & \textbf{24.77}  & 103.62 & $847.44^{\ddagger}$ \\ \specialrule{.00002em}{.1em}{.1em}
    \multirow{4}{*}{$5.26\cdot 10^5$} & $10^{-2}$ & 12(12)[12]\{144\} & 14(1349) & $4000^{\ddagger}$  & \textbf{116.31} & 472.25 & $5120.49^{\ddagger}$ \\
  &   $10^{-4}$ & 12(12)[12]\{144\} & 14(1309) &  $4000^{\ddagger}$  &  \textbf{121.45} & 477.11 &   $5336.74^{\ddagger}$\\
  &   $10^{-6}$ &12(12)[12]\{144\} & 14(1309) & $4000^{\ddagger}$  & \textbf{115.93}  & 491.46&   $5225.83^{\ddagger}$\\ 
    &   $0$ & 12(12)[12]\{144\}  &14(1364)  & $4000^{\ddagger}$  &   \textbf{116.24} & 508.72 &  $5047.28^{\ddagger}$ \\ \specialrule{.00002em}{.1em}{.1em}
    \multirow{4}{*}{$2.11\cdot 10^6$} & $10^{-2}$ & 13(13)[13]\{160\} & $\dagger$\tnote{2}  & $\dagger$ &  \textbf{604.55} & $\dagger$ & $\dagger$ \\
  &   $10^{-4}$ & 13(13)[13]\{160\}  & $\dagger$ &  $\dagger$ & \textbf{605.09} & $\dagger$ & $\dagger$\\
  &   $10^{-6}$ &  13(13)[13]\{160\} & $\dagger$ & $\dagger$ &  \textbf{615.46} & $\dagger$& $\dagger$\\ 
    &   $0$ &   13(13)[13]\{160\} & $\dagger$ & $\dagger$ & \textbf{608.02}  & $\dagger$ & $\dagger$ \\
\specialrule{.2em}{.05em}{.05em} 
\end{tabular}
\begin{tablenotes}
\item[1] $\ddagger$ indicates that the solver reached the maximum number of iterations.
\item[2] $\dagger$ indicates that the solver ran out of memory.
\end{tablenotes}
\end{threeparttable}}
\end{table}
\par Some observations are in order. Firstly, we should note that the active-set and the IP-PMM algorithms are very efficient for finding a solution to relatively high accuracy (i.e. $\texttt{tol} = 10^{-6}$). Their convergence behaviour is barely affected by the $L^1$ regularization parameter, and both solvers exhibit robustness with respect to the problem size. The active-set scheme consistently outperforms both IP-PMM and OSQP, and manages to solve the largest instances without running into memory issues. We observe that OSQP is not particularly efficient or robust, failing to solve most instances to the desired accuracy. IP-PMM is rather efficient, but is outperformed by the active-set scheme due to the inherent ill-conditioning of its associated linear systems. Surprisingly, for this set of problems, and for the requested tolerance, the active-set scheme requires a comparable number of Newton iterations compared to those required by IP-PMM, and thus performs fewer Krylov iterations (since its associated linear systems are much better conditioned). It requires less memory than IP-PMM (since it only considers an active-set at each iteration), which is verified by the inability of the latter to solve the largest instances. Next, we fix $\alpha_1 = 10^{-4}$ and $n = 1.32\cdot 10^{5}$, and vary the $L^2$ regularization parameter as well as the tolerance. The results are collected in Table \ref{Table Poisson optimal control: varying tolerance and L2 regularization}.
\begin{table}[!ht]
\caption{Poisson control: varying accuracy and $L^2$ regularization. $(n,\alpha_1) = (1.32\cdot 10^5,10^{-4})$.\label{Table Poisson optimal control: varying tolerance and L2 regularization}}
\centering
\scalebox{0.90}{\begin{threeparttable}
\begin{tabular}{llllllll}     
\specialrule{.2em}{.05em}{.05em} 
    \multirow{2}{*}{\textbf{{tol}}} & \multirow{2}{*}{$\bm{\alpha_2}$} & \multicolumn{3}{c}{\textbf{Iterations}}   & \multicolumn{3}{c}{\textbf{Time (s)}}  \\   \cmidrule(l{2pt}r{2pt}){3-5} \cmidrule(l{2pt}r{2pt}){6-8}
&   &  {{PMM(SSN)[Fact.]\{Krylov\}}} & {{IP-PMM(Krylov)}}  & {{OSQP}} & AS & IP-PMM & OSQP     \\ \specialrule{.1em}{.3em}{.3em}
\multirow{4}{*}{$10^{-3}$} & $10^{-2}$ & 1(1)[1]\{11\} & 6(93)  &  25 &  \textbf{12.37} & 15.38 & 42.70 \\
  &   $10^{-4}$ & 1(1)[1]\{11\} & 6(74) &  25 &   \textbf{12.11} & 16.16 & 41.00  \\
  &   $10^{-6}$ & 1(1)[1]\{11\} & 5(46)  & 25 & 12.39  & \textbf{12.14} & 40.44\\ 
    &   $0$ & 1(1)[1]\{11\} & 5(46) & 25 & 11.92 & \textbf{11.91} & 12.45 \\  \specialrule{.00002em}{.1em}{.1em}
\multirow{4}{*}{$10^{-5}$} & $10^{-2}$  & 7(7)[7]\{91\} &  8(223) & 50 & \textbf{19.71} & 26.01 & 45.37\\
  &   $10^{-4}$ & 7(7)[7]\{91\} & 8(198) & 50 &  \textbf{19.66} & 26.67 & 47.29 \\
  &   $10^{-6}$ & 7(7)[7]\{91\} & 7(136)  &  50 &  \textbf{19.39} & 20.30 & 44.99\\ 
    &   $0$ & 7(7)[7]\{91\} & 7(136) & 50 &  19.29 & 19.32 & \textbf{15.90} \\ \specialrule{.00002em}{.1em}{.1em}
\multirow{4}{*}{$10^{-7}$} & $10^{-2}$ & 20(55)[47]\{547\} & 15(1692) & $4000^{\ddagger}$\tnote{1}  &  \textbf{62.84} & 129.51 & $848.35^{\ddagger}$ \\
  &   $10^{-4}$ & 19(62)[52]\{610\} & 15(937)  & $4000^{\ddagger}$ & \textbf{71.61} & 85.65& $845.62^{\ddagger}$\\
  &   $10^{-6}$ & 20(68)[55]\{655\} & 16(933)  & $4000^{\ddagger}$ &  \textbf{73.65} & 81.02& $829.25^{\ddagger}$ \\
    &   $0$ & 20(64)[51]\{631\} & 16(824)  & $4000^{\ddagger}$ & \textbf{67.40} & 71.81 & $275.46^{\ddagger}$\\  
\specialrule{.2em}{.05em}{.05em} 
\end{tabular}\begin{tablenotes}
\item[1] $\ddagger$ indicates that the solver reached the maximum number of iterations.
\end{tablenotes}
\end{threeparttable}}
\end{table}
\par The results in Table \ref{Table Poisson optimal control: varying tolerance and L2 regularization} indicate that, unlike OSQP and IP-PMM, the active-set method exhibits a good level of robustness with respect to the $L^2$ regularization parameters. Nonetheless, both second-order solvers provide a solution reliably for lower tolerance values. When requesting a (relatively) low-accuracy solution (i.e. $\texttt{tol} =10^{-3}$ or $\texttt{tol} = 10^{-5}$), we observe that the second-order solver is barely needed, as the starting point yielded by Algorithm \ref{proximal ADMM algorithm} is already very close to such a solution. This is also verified by the good behaviour of OSQP in Table \ref{Table Poisson optimal control: varying tolerance and L2 regularization} for 3- or 5-digit accurate solutions. When requesting a highly accurate solution (i.e. $\texttt{tol} = 10^{-7}$) we observe that the number of SSN iterations performed by the active-set method is much greater than the number of IP-PMM iterations. Nevertheless, the former solver performs fewer Krylov iterations, which confirms that its associated linear systems are much better conditioned. We note that the AS implementation is rather aggressive, since we allow at most 8 SSN iterations per PMM sub-problem. Overall, the method scales quite well with the size of the problem, and the memory requirements are very reasonable, allowing for the solution of large-scale instances on a personal computer.  
\subsubsection{Convection--diffusion optimal control}
\par We now consider the optimal control of the convection--diffusion equation, i.e. $-\epsilon \rm{\Delta y} + \rm{w} \cdot \nabla y = u$, on the domain $\Omega = (0,1)^2$, where $\rm{w}$ is the wind vector given by $\rm{w} = [2x_2(1-x_1)^2, -2x_1(1-x_2^2)]^\top$, with control bounds $\rm{u_a} = -2$, $\rm{u_b} = 1.5$ and free state (e.g. see \cite[Section 5.2]{NLAA:PearsonPorcStoll}). The problem is discretized using Q1 finite elements, employing the Streamline Upwind Petrov-Galerkin (SUPG) upwinding scheme implemented in \cite{BrooksHughesCMAME}. We set the desired state as $\rm{\bar{y}} = exp(-64((x_1 - 0.5)^2 + (x_2 - 0.5)^2))$, with zero boundary conditions, and the diffusion coefficient as $\epsilon = 0.05$. In Table \ref{Table Convection-diffusion optimal control: varying grid-size and L1 regularization}, we fix the ${L}^2$ regularization parameter as $\alpha_2 = 10^{-2}$ and the tolerance to $\texttt{tol} = 10^{-6}$ and run the three methods with different ${L}^1$ regularization values (i.e. $\alpha_1$) and with increasing grid size. 
\begin{table}[!ht]
\caption{Convection--diffusion control: varying grid size and $L^1$ regularization. ($\texttt{tol}, \alpha_2,\epsilon) = (10^{-6},10^{-2},0.05)$.\label{Table Convection-diffusion optimal control: varying grid-size and L1 regularization}}
\centering
\scalebox{0.84}{\begin{threeparttable}
\begin{tabular}{llllllll}     
\specialrule{.2em}{.05em}{.05em}
    \multirow{2}{*}{\textbf{$\bm{n}$}} & \multirow{2}{*}{$\bm{\alpha_1}$} & \multicolumn{3}{c}{\textbf{Iterations}}     & \multicolumn{3}{c}{\textbf{Time (s)}}  \\   \cmidrule(l{2pt}r{2pt}){3-5} \cmidrule(l{2pt}r{2pt}){6-8}
&   &  {{PMM(SSN)[Fact.]\{Krylov\}}} & {{IP-PMM(Krylov)}}  & {{OSQP}} & AS & IP-PMM & OSQP  \\ \specialrule{.1em}{.3em}{.3em}
 \multirow{4}{*}{$8.45\cdot 10^3$} & $10^{-3}$ & 16(70)[34]\{2430\} & 26(8085) & 200 & 9.56 &33.55 & \textbf{1.38}\\
  &   $10^{-4}$ & 16(56)[25]\{1750\} & 25(7546) & 150 &  7.31 & 32.31 & \textbf{1.14}\\
  &   $10^{-5}$ & 16(43)[24]\{997\} & 25(7327) & 975 & \textbf{4.45} & 27.37 & 5.61 \\ 
    &   $0$ & 16(43)[24]\{787\} & 25(7509) & 150 & 4.15 & 27.92 & \textbf{1.11}\\ \specialrule{.00002em}{.1em}{.1em}
 \multirow{4}{*}{$3.32\cdot 10^4$} & $10^{-3}$ & 17(50)[29]\{1351\} & 25(7331) & 475 & 22.22  &104.77 & \textbf{19.31}\\
  &   $10^{-4}$ & 17(38)[26]\{777\} & 25(7198) & 325 & 14.87 & 99.49 & \textbf{14.21}\\
  &   $10^{-5}$ & 17(38)[27]\{678\} & 25(7176) & 1150 & \textbf{13.69} & 102.91 & 41.01 \\ 
    &   $0$ & 17(38)[27]\{600\} & 25(7145) & 325 & \textbf{12.75} & 128.95  & 14.41\\ \specialrule{.00002em}{.1em}{.1em}
\multirow{4}{*}{$1.32\cdot 10^5$} & $10^{-3}$ & 20(36)[27]\{865\} & 23(5544) & 3125 & \textbf{65.17} & 399.20 & 576.45 \\
  &   $10^{-4}$ & 20(31)[25]\{577\} & 22(4971) & 3125 &  \textbf{52.97} & 356.23 & 566.14 \\
  &   $10^{-5}$ & 20(31)[28]\{521\} & 23(5428) & 3075 & \textbf{54.49} & 394.42 & 558.23 \\ 
    &   $0$ & 20(31)[26]\{490\} & 23(5414) & 3075 & \textbf{51.82} & 392.11 & 557.86\\ \specialrule{.00002em}{.1em}{.1em}
    \multirow{4}{*}{$5.26\cdot 10^5$} & $10^{-3}$ & 23(24)[23]\{330\} & 9(385) & $4000^{\ddagger}$\tnote{1} & \textbf{144.88} & 161.23 & $5002.59^{\ddagger}$ \\
  &   $10^{-4}$ & 23(24)[24]\{325\} & 9(385) & $4000^{\ddagger}$ & \textbf{144.76} & 162.42 & $5329.52^{\ddagger}$ \\
  &   $10^{-5}$ & 23(24)[24]\{318\} & 9(385) & $4000^{\ddagger}$   & \textbf{154.46} & 166.53 & $5147.38^{\ddagger}$ \\ 
    &   $0$ & 23(24)[23]\{321\} & 9(385) & $4000^{\ddagger}$  & \textbf{165.87} & 166.53 & $4984.27^{\ddagger}$\\ \specialrule{.00002em}{.1em}{.1em}
        \multirow{4}{*}{$2.11\cdot 10^6$} & $10^{-3}$ & 41(44)[41]\{450\} & $\dagger$\tnote{2} & $\dagger$ & \textbf{964.01} &  $\dagger$ &  $\dagger$ \\
  &   $10^{-4}$ & 41(44)[41]\{449\} & $\dagger$ &  $\dagger$ & \textbf{981.50} &  $\dagger$ &  $\dagger$ \\
  &   $10^{-5}$ & 41(44)[44]\{456\} &  $\dagger$ &  $\dagger$   & \textbf{999.31} &  $\dagger$ &  $\dagger$  \\ 
    &   $0$ & 41(44)[41]\{457\} &  $\dagger$ &  $\dagger$  & \textbf{1086.92} &  $\dagger$ &  $\dagger$\\
\specialrule{.2em}{.05em}{.05em}
\end{tabular}\begin{tablenotes}
\item[1] $\ddagger$ indicates that the solver reached the maximum number of iterations.
\item[2] $\dagger$ indicates that the solver ran out of memory.
\end{tablenotes}
\end{threeparttable}}
\end{table}
\par From Table \ref{Table Convection-diffusion optimal control: varying grid-size and L1 regularization} we observe, similar to the Poisson examples, that OSQP is competitive for the small instances, but is unable to solve the larger instances and does not exhibit robustness with respect to the $L^1$ regularization parameter. In contrast, both second-order solvers (active-set and IP-PMM) exhibit robustness with respect to the $L^1$-regularization parameter. However, IP-PMM seems to be affected by the problem size. Indeed, the behaviour of IP-PMM is significantly worse compared to that of the active-set method for the smaller instances (with the latter being up to 6 times faster on some instances), which are the most ill-conditioned ones. Overall, the active-set scheme outperforms the other two methods in all the large instances, and scales better in terms of memory requirements. Next, we set $\alpha_1 = 10^{-3}$, $n = 1.32 \cdot 10^{5}$, $\texttt{tol} = 10^{-6}$, and run the method with varying $L^2$ regularization and diffusion coefficient $\epsilon$. The results are collected in Table \ref{Table convection diffusion optimal control: varying diffusion and L2 regularization}.
\begin{table}[!ht]
\caption{Convection--diffusion control: varying diffusion and $L^2$ regularization. $(n,\texttt{tol},\alpha_1) = (1.32\cdot 10^5,10^{-6},10^{-3})$.\label{Table convection diffusion optimal control: varying diffusion and L2 regularization}}
\centering
\scalebox{0.90}{\begin{threeparttable}
\begin{tabular}{llllllll}     
\specialrule{.2em}{.05em}{.05em}
    \multirow{2}{*}{$\bm{\epsilon}$} & \multirow{2}{*}{$\bm{\alpha_2}$} & \multicolumn{3}{c}{\textbf{Iterations}}      & \multicolumn{3}{c}{\textbf{Time (s)}}  \\   \cmidrule(l{2pt}r{2pt}){3-5} \cmidrule(l{2pt}r{2pt}){6-8}
&   &  {{PMM(SSN)[Fact.]\{Krylov\}}} & {{IP-PMM(Krylov)}}  & {{OSQP}} & AS & IP-PMM & OSQP  \\ \specialrule{.1em}{.3em}{.3em}
\multirow{4}{*}{$0.01$} & $10^{-2}$ & 21(44)[33]\{949\} & 23(5674) & 400 & \textbf{72.33} & 419.83  & 128.24 \\
  &   $10^{-4}$ & 22(56)[41]\{1155\} & 21(2009) & 1100 & \textbf{91.38} & 159.13 & 227.03\\
  &   $10^{-6}$ & 22(56)[40]\{1172\} & 20(1623) & 1150 & \textbf{87.11}  & 129.23 & 246.03 \\ 
    &   $0$ & 22(56)[40]\{1191\} &  20(1381) & 1500  & \textbf{87.65}  & 109.90 & 102.79\\  \specialrule{.00002em}{.1em}{.1em}
\multirow{4}{*}{$0.02$} & $10^{-2}$  &20(38)[29]\{901\}  & 23(5508) & 875 & \textbf{70.20}   & 411.26& 184.55 \\
  &   $10^{-4}$ & 22(56)[38]\{1199\} &  20(1851) & 1425 & \textbf{86.07}  & 150.27 & 279.97\\
  &   $10^{-6}$ & 22(56)[43]\{1177\} & 20(1636)  & 2000 & \textbf{87.01} & 136.57 & 377.83\\ 
    &   $0$ & 22(64)[47]\{1389\} & 20(1399) & 1500 &  99.41 & 113.05 & \textbf{93.30} \\  \specialrule{.00002em}{.1em}{.1em}
\multirow{4}{*}{$0.05$} & $10^{-2}$ & 20(36)[27]\{865\} & 23(5544)  & 3125 & \textbf{65.17} & 399.20 & 576.45 \\
  &   $10^{-4}$ & 21(43)[34]\{916\} & 19(1692) &  3000 &  \textbf{71.85} & 136.56&  550.13\\
  &   $10^{-6}$ & 21(43)[33]\{957\} &  18(1396) &  3375 & \textbf{74.79} & 108.95  &  638.95\\ 
    &   $0$ & 21(43)[33]\{957\} & 18(1195) &  2975 & \textbf{73.41} & 98.43& 179.38 \\  
\specialrule{.2em}{.05em}{.05em}
\end{tabular}
\end{threeparttable}}
\end{table}
\par Again, AS is quite robust with respect to the $L^2$ regularization parameter, which is not the case for the other two solvers. Furthermore, the same applies for the convection diffusion coefficient $\epsilon$, although IP-PMM is also little affected by it. This is not the case for OSQP, which exhibits a very different behaviour for different values of the diffusion coefficient. Finally, we can observe that the active-set method is able to find accurate solutions consistently and very efficiently, making it a competitive solver for PDE-constrained optimization instances. Overall, we observe that the proposed scheme is consistently more efficient and reliable than the other two methods and has the ability to provide highly accurate solutions without running into numerical issues.
\par Nevertheless, we should mention that we expect IP-PMM to behave better for arbitrary convex quadratic instances, since, in general, interior point methods are more robust solvers (both theoretically and numerically). However, in certain cases where $\ell_1$ terms are present in the objective, the proposed active-set scheme can be a much better choice in terms of stability and efficiency. This has been numerically demonstrated here for the case of certain regularized linear regression and $L^1$-regularized PDE-constrained optimization problems, but we conjecture that this behaviour can be observed for several other problems appearing in practice. A more in-depth study comparing the three schemes on various other important applications, including problems with general piecewise-linear structure and nonseparable $\ell_1$-terms, has been written in parallel with this work and has been compiled in an accompanying paper.
\section{Conclusions} \label{sec: Conclusions}
\par In this paper we derived an efficient active-set method suitable for the solution of $\ell_1$-regularized convex quadratic instances. The algorithm consists of a proximal method of multipliers that employs a standard semismooth Newton method for solving the associated sub-problems. We have shown that the proposed PMM converges globally under very mild assumptions, while it can potentially achieve a global linear and local superlinear convergence rate. The linear systems within SSN are solved using the preconditioned minimum residual method, and the proposed preconditioner is cheap to invert and exhibits very good behaviour and robustness with respect to the PMM penalty parameters. The efficiency of the method is further improved by using a warm-starting strategy based on a proximal alternating direction method of multipliers. The proposed approach has been extensively tested on certain regularized linear regression and PDE-constrained optimization problems, and computational evidence, including a detailed comparison against an IPM and an ADMM solver, has been provided to demonstrate its efficiency, reliability, and scalability. An accompanying work extending the proposed methodology to problems with general piecewise-linear terms in the objective has been written in parallel with this paper.
\appendix
\section{Appendix}
\subsection{Derivation of the dual problem} \label{Appendix: derivation of dual}
\par The dual of \eqref{primal problem} is $ \sup_{y,z}\inf_x \left\{ \ell(x,y,z)\right\}.$ Let $f(x) = c^\top x + \frac{1}{2}x^\top Q x$. We have
\begin{equation*}
\begin{split}
\inf_{x}\left\{\ell(x,y,z)\right\} =&\ -\sup_x \left\{(y^\top A x-z)^\top x -\left(f(x) + g(x) \right) \right\} + y^\top b - \delta_{\mathcal{K}}^*(z) \\
=&\ -\left(f + g\right)^*(A^\top y - z) + y^\top b - \delta_{\mathcal{K}}^*(z)\\
=&\ -\inf_{x'} \left\{f^*\left(A^\top y - z - x'\right) + g^*(x') \right\} + y^\top b - \delta_{\mathcal{K}}^*(z),
\end{split}
\end{equation*}
\noindent where we used the definition of the convex conjugate and a property of the infimal convolution, i.e. $(f + g)^*(x) = \inf_{x'}\left\{f^*(x - x') + g^*(x')\right\}$ (see \cite[Proposition 13.21]{Springer:BausComb}). However, from the definition of $f(\cdot)$ we have $f^*(A^\top y - z - x') = \frac{1}{2}x^\top Q x + \delta_{\{0\}}\left(c + Qx - A^\top y + z + x'\right).$ By substituting this, and by eliminating variable $x'$ we obtain \eqref{dual problem}.
\iffalse
\subsection{Characterization of {\boldmath $\textbf{dist}(0,F_{\beta_k,\rho_k}(x,y))$}} \label{Appendix: characterization of distance}
\par Let an arbitrary pair $(x,y)$ be given, and define
\[F_{\beta_k,\rho_k}(x,y) \coloneqq \left\{(u',v') \colon u' \in r_{\beta_k,\rho_k}(x,y) + \partial g(x),\colon v' = Ax + \beta_k^{-1}(y-y_k) - b \right\}, \]
\noindent where $r_{\beta_k,\rho_k}(x,y)$ is defined in Section \ref{sec: PAL penalties}. Using the definition of $\textnormal{dist}(x,\mathcal{A})$, for some closed convex set $\mathcal{A}$,
\begin{equation*}
\textnormal{dist}\left(0,F_{\beta_k,\rho_k}(x,y)\right) = \left\|\begin{bmatrix}
\textbf{prox}_{\hat{f}}(0)\\
Ax + \beta_k^{-1}(y-y_k) - b
\end{bmatrix}\right\|
\end{equation*}
\noindent where $\hat{f}(w)  = \delta_{\partial g(x)}(w - r_{\beta_k,\rho_k})$, and $\delta_{\partial g(x)}(\cdot)$ is an indicator function of the subdifferential of $g(\cdot)$. Then, we note that $\textbf{prox}_{\hat{f}}(0) = \textbf{prox}_{\delta_{\partial g(x)}}(-r_{\beta_k,\rho_k}) + r_{\beta_k,\rho_k}$. A direct evaluation of this proximal operator yields the characterization used in Algorithm \ref{primal-dual PMM algorithm}.
\fi
\subsection{Termination criteria} \label{Appendix: termination criteria}
\par We write the optimality conditions for \eqref{primal problem}--\eqref{dual problem} as
\begin{equation} \label{optimality conditions for P-D}
x = \textbf{prox}_{g}\left(x -c - Qx + A^\top y - z\right),\qquad Ax = b,\qquad x = \Pi_{\mathcal{K}}(x + z),
\end{equation}
\noindent and the termination criteria for Algorithm \ref{primal-dual PMM algorithm} (given a tolerance $\epsilon > 0$) are set as
\begin{equation} \label{termination criteria for PD-PMM}
\frac{\|x - \textbf{prox}_{g}\left(x -c - Qx + A^\top y - z\right)\|}{1+\|c\|_{\infty}} \leq \epsilon,\quad \frac{\|Ax - b\|}{1+\|b\|_{\infty}} \leq \epsilon,\quad \frac{\|x - \Pi_{\mathcal{K}}(x + z)\|}{1+\|x\|_{\infty} + \|z\|_{\infty}} \leq \epsilon.
\end{equation}
\par Finally, the termination criteria of Algorithm \ref{proximal ADMM algorithm} are set as
\begin{equation} \label{termination criteria for pADMM}
\frac{\left\|c + Qx - A^\top y_1 + y_2\right\|}{1+\|c\|} \leq \epsilon,\quad \frac{\left\|\left(Ax - b,w-x\right)\right\|}{1+\|b\|} \leq \epsilon, \quad \frac{\|w - \Pi_{\mathcal{K}}\left(\textbf{prox}_{g}\left(w + y_2\right)\right)\|}{1+\|w\| + \|y_2\|} \leq \epsilon.
\end{equation}
\bibliography{references} 

\begin{thebibliography}{10}

\bibitem{OMS:GondAlt}
{\sc A.~Altman and J.~Gondzio}, {\em {R}egularized symmetric indefinite systems
  in interior point methods for linear and quadratic optimization},
  Optimization Methods and Software, 11 (1999), pp.~275--302,
  \url{https://doi.org/10.1080/10556789908805754}.

\bibitem{Springer:BausComb}
{\sc H.~H. Bauschke and P.~L. Combettes}, {\em {C}onvex {A}nalysis and
  {M}onotone {O}perator {T}heory in {H}ilbert {S}paces}, CMS Books in
  Mathematics, Springer, New York, NY, 2011,
  \url{https://doi.org/10.1007/978-1-4419-9467-7}.

\bibitem{SIAM:Beck}
{\sc A.~Beck}, {\em {F}irst-{O}rder {M}ethods in {O}ptimization}, MOS-SIAM
  Series on Optimization, SIAM \& Mathematical Optimization Society,
  Philadelphia, 2017, \url{https://doi.org/10.1137/1.9781611974997}.

\bibitem{NLAA:BergGondMartPearPoug}
{\sc L.~Bergamaschi, J.~Gondzio, A.~Martínez, J.~W. Pearson, and
  S.~Pougkakiotis}, {\em {A} new preconditioning approach for an interior
  point-proximal method of multipliers for linear and convex quadratic
  programming}, Numerical Linear Algebra with Applications, 28 (2020),
  p.~e2361, \url{https://doi.org/10.1002/nla.2361}.

\bibitem{BertsekasNedicOzdaglar}
{\sc D.~P. Bertsekas, A.~Nedic, and E.~Ozdaglar}, {\em {C}onvex {A}nalysis and
  {O}ptimization}, Athena Scientific, 2003.

\bibitem{JCAM:BoggsTolleSQP}
{\sc P.~T. Boggs and J.~W. Tolle}, {\em {S}equential quadratic programming for
  large-scale nonlinear optimization}, Journal of Computational and Applied
  Mathematics, 124 (2000), pp.~123--137,
  \url{https://doi.org/10.1016/S0377-0427(00)00429-5}.

\bibitem{BrooksHughesCMAME}
{\sc A.~N. Brooks and T.~J.~R. Hughes}, {\em {S}treamline
  upwind/{P}etrov--{G}alerkin formulations for convection dominated flows with
  particular emphasis on the incompressible {N}avier--{S}tokes equations},
  Computer Methods in Applied Mechanics and Engineering, 32 (1982),
  pp.~199--259, \url{https://doi.org/10.1016/0045-7825(82)90071-8}.

\bibitem{CC01a}
{\sc C.-C. Chang and C.-J. Lin}, {\em {LIBSVM}: A library for support vector
  machines}, ACM Transactions on Intelligent Systems and Technology, 2 (2011),
  pp.~27:1--27:27.
\newblock Software available at \url{http://www.csie.ntu.edu.tw/~cjlin/libsvm}.

\bibitem{NLAA:ChenQi}
{\sc J.~Chen and L.~Qi}, {\em {G}lobally and superlinearly convergent inexact
  {N}ewton-{K}rylov algorithms for solving nonsmooth equations}, Numerical
  Linear Algebra with Applications, 17 (2010), pp.~155--174,
  \url{https://doi.org/10.1002/nla.673}.

\bibitem{SIREV:Chenetal}
{\sc S.~S. Chen, D.~L. Donoho, and M.~A. Saunders}, {\em {A}tomic decomposition
  by basis pursuit}, SIAM Review, 43 (2001), pp.~129--159,
  \url{https://doi.org/10.1137/S003614450037906X}.

\bibitem{SIAMOpt:Christofetal}
{\sc C.~Christof, H.~C. De~Los~Reyes, and C.~Meyer}, {\em {A} nonsmooth
  trust-region method for locally {L}ipschitz functions with applications to
  optimization problems constrained by variational inequalities}, SIAM Journal
  on Optimization, 30 (2020), pp.~2163--2196,
  \url{https://doi.org/10.1137/18M1164925}.

\bibitem{JWS:Clarke}
{\sc F.~Clarke}, {\em {O}ptimization and {N}onsmooth {A}nalysis}, Classics in
  Applied Mathematics, John Wiley and Sons, New York, 1990,
  \url{https://doi.org/10.1137/1.9781611971309}.

\bibitem{arXiv:ClasValk}
{\sc C.~Clason and T.~Valkonen}, {\em {I}ntroduction to {N}onsmooth {A}nalysis
  and {O}ptimization}, arXiv preprint arXiv:1912.08672,  (2020).

\bibitem{COAP:Marchi}
{\sc A.~De~Marchi}, {\em {O}n a primal-dual {N}ewton proximal method for convex
  quadratic programs}, Computational Optimization and Applications,  (2022),
  \url{https://doi.org/10.1007/s10589-021-00342-y}.

\bibitem{SIREV:DeSimone_etal}
{\sc V.~De~Simone, D.~di~Serafino, J.~Gondzio, S.~Pougkakiotis, and M.~Viola},
  {\em {S}parse approximations with interior point methods}, SIAM Review, 64
  (2022), pp.~954--988, \url{https://doi.org/10.1137/21M1401103}.

\bibitem{SciComp:DengYin}
{\sc W.~Deng and W.~Yin}, {\em {O}n the global and linear convergence of the
  generalized alternating direction method of multipliers}, Journal of
  Scientific Computing, 66 (2016), pp.~889--916,
  \url{https://doi.org/10.1007/s10915-015-0048-x}.

\bibitem{MathProg:Dennis_etal}
{\sc J.~E. Dennis, S.-B.~B. Li, and R.~A. Tapia}, {\em {A} unified approach to
  global convergence of trust region methods for nonsmooth optimization},
  Mathematical Programming, 68 (1995), pp.~319--346,
  \url{https://doi.org/10.1007/BF01585770}.

\bibitem{IEEE_CDC:Dhingra_etal}
{\sc N.~K. Dhingra, S.~Z. Khong, and M.~R. Jovanović}, {\em A second order
  primal-dual algorithm for nonsmooth convex composite optimization}, in 2017
  IEEE 56th Annual Conference on Decision and Control (CDC), 2017,
  pp.~2868--2873, \url{https://doi.org/10.1109/CDC.2017.8264075}.

\bibitem{Springer:DonRock}
{\sc A.~L. Dontchev and R.~T. Rockafellar}, {\em {I}mplicit {F}unctions and
  {S}olution {M}appings}, Springer Series in Operations Research and Financial
  Engineering, Springer, New York, NY, 2014,
  \url{https://doi.org/10.1007/978-1-4939-1037-3}.

\bibitem{IFISSACM}
{\sc H.~C. Elman, A.~Ramage, and D.~J. Silvester}, {\em {A}lgorithm 866:
  {IFISS}, a {M}atlab toolbox for modelling incompressible flow}, ACM
  Transactions on Mathematical Software, 33 (2007), p.~14,
  \url{https://doi.org/10.1145/1236463.1236469}.

\bibitem{IFISSSIAMREVIEW}
{\sc H.~C. Elman, A.~Ramage, and D.~J. Silvester}, {\em {IFISS}: {A}
  computational laboratory for investigating incompressible flow problems},
  SIAM Review, 52 (2014), pp.~261--273,
  \url{https://doi.org/10.1137/120891393}.

\bibitem{IPMs:FountoulakisEtAl2013}
{\sc K.~Fountoulakis, J.~Gondzio, and P.~Zhlobich}, {\em Matrix-free interior
  point method for compressed sensing problems}, Mathematical Programming
  Computation, 6 (2014), pp.~1--31,
  \url{https://doi.org/10.1007/s12532-013-0063-6}.

\bibitem{MathProgComp:FriedOrban}
{\sc M.~P. Friedlander and D.~Orban}, {\em {A} primal-dual regularized
  interior-point method for convex quadratic progams}, Mathematical Programming
  Computation, 4 (2012), pp.~71--107,
  \url{https://doi.org/10.1007/s12532-012-0035-2}.

\bibitem{COAP:GillRobi}
{\sc P.~E. Gill and D.~P. Robinson}, {\em {A} primal--dual augmented
  {L}agrangian}, Computational Optimization and Applications, 15 (2012),
  pp.~1--25, \url{https://doi.org/10.1007/s10589-010-9339-1}.

\bibitem{arXiv:GondPougkPears}
{\sc J.~Gondzio, S.~Pougkakiotis, and J.~W. Pearson}, {\em General-purpose
  preconditioning for regularized interior point methods}, Computational
  Optimization and Applications, 83 (2022), pp.~727--757,
  \url{https://doi.org/10.1007/s10589-022-00424-5}.

\bibitem{MathOR:Han}
{\sc S.-P. Han, J.-S. Pang, and N.~Rangaraj}, {\em {G}lobally convergent
  {N}ewton methods for nonsmooth equations}, Mathematics of Operations
  Research, 17 (1992), pp.~586--607,
  \url{https://doi.org/10.1287/moor.17.3.586}.

\bibitem{InverseProbs:HansRaasch}
{\sc E.~Hans and T.~Raasch}, {\em {G}lobal convergence of damped semismooth
  {N}ewton methods for $\ell_1$ {T}ikhonov regularization}, Inverse Problems,
  31 (2015), p.~025005, \url{https://doi.org/10.1088/0266-5611/31/2/025005}.

\bibitem{MathProgComp:Hermans_etal}
{\sc B.~Hermans, A.~Themelis, and P.~Patrinos}, {\em {QPALM}: a proximal
  augmented {L}agrangian method for nonconvex quadratic programs}, Mathematical
  Programming Computation, 14 (2022), pp.~497--541,
  \url{https://doi.org/10.1007/s12532-022-00218-0}.

\bibitem{HestenesSteifelPCG}
{\sc M.~R. Hestenes and E.~Stiefel}, {\em {M}ethod of conjugate gradients for
  solving linear systems}, Journal of Research of the National Bureau of
  Standards, 49 (1952), pp.~409--436.

\bibitem{ApplMathOpt:HirStroNgu}
{\sc J.-B. Hiriart-Urruty, J.-J. Strodiot, and V.~H. Nguyen}, {\em
  {G}eneralized {H}essian matrix and second-order optimality conditions for
  problems with {$C^{1,1}$} data}, Applied Mathematics and Optimization, 11
  (1984), pp.~43--56, \url{https://doi.org/10.1007/BF01442169}.

\bibitem{MathProg:Itoetal}
{\sc K.~Ito and K.~Kunnisch}, {\em {O}n a semi-smooth {N}ewton method and its
  globalization}, Mathematical Programming, 118 (2009), pp.~347--370,
  \url{https://doi.org/10.1007/s10107-007-0196-3}.

\bibitem{SIAMOPT:Leeetal}
{\sc J.~D. Lee, Y.~Sun, and M.~A. Saunders}, {\em {P}roximal {N}ewton-type
  methods for minimizing composite functions}, SIAM Journal on Optimization, 24
  (2014), pp.~1420--1443, \url{https://doi.org/10.1137/130921428}.

\bibitem{SIAMOpt:Lietal2}
{\sc X.~Li, D.~Sun, and K.~C. Toh}, {\em {A} highly efficient semismooth
  {N}ewton augmented {L}agrangian method for solving {L}asso problems}, SIAM
  Journal on Optimization, 28 (2018), pp.~433--458,
  \url{https://doi.org/10.1137/16M1097572}.

\bibitem{SIAMOpt:Lietal}
{\sc X.~Li, D.~Sun, and K.~C. Toh}, {\em {A}n asymptotically superilinearly
  convergent semismooth {N}ewton augmented {L}agrangian method for linear
  programming}, SIAM Journal on Optimization, 30 (2020), pp.~2410--2440,
  \url{https://doi.org/10.1137/19M1251795}.

\bibitem{OptEng:MannelRund}
{\sc F.~Mannel and A.~Rund}, {\em {A} hybrid semismooth quasi-{N}ewton method
  for nonsmooth optimal control with {PDEs}}, Optimization and Engineering, 22
  (2021), pp.~2087--2125, \url{https://doi.org/10.1007/s11081-020-09523-w}.

\bibitem{CAM:MartQi}
{\sc J.~Martínez and L.~Qi}, {\em {I}nexact {N}ewton methods for solving
  nonsmooth equations}, Journal of Computational and Applied Mathematics, 60
  (1995), pp.~127--145, \url{https://doi.org/10.1016/0377-0427(94)00088-I}.

\bibitem{SIAMOpt:Mehrotra}
{\sc S.~Mehrotra}, {\em On the implementation of a primal-dual interior point
  method}, SIAM Journal on Optimization, 2 (1992), pp.~575--601,
  \url{https://doi.org/10.1137/0802028}.

\bibitem{BSMF:Moreau}
{\sc J.~J. Moreau}, {\em {P}roximité et dualité dans un espace {H}ilbertien},
  Bulletin de la Société Mathématique de France, 93 (1965), pp.~273--299,
  \url{https://doi.org/10.24033/bsmf.1625}.

\bibitem{PaigeSaundersSIAMNumAnal}
{\sc C.~C. Paige and M.~A. Saunders}, {\em {S}olution of sparse indefinite
  systems of linear equations}, SIAM Journal on Numerical Analysis, 12 (1975),
  pp.~617--629, \url{https://doi.org/10.1137/0712047}.

\bibitem{IEEE_DC:Patrinos_etal}
{\sc P.~Patrinos and A.~Bemporad}, {\em {P}roximal {N}ewton methods for convex
  composite optimization}, in 52nd IEEE Conference on Decision and Control,
  2013, pp.~2358--2363, \url{https://doi.org/10.1109/CDC.2013.6760233}.

\bibitem{NLAA:PearsonPorcStoll}
{\sc J.~W. Pearson, M.~Porcelli, and M.~Stoll}, {\em {I}nterior-point methods
  and preconditioning for {PDE}-constrained optimization problems involving
  sparsity terms}, Numerical Linear Algebra with Applications, 27 (2019),
  p.~e2276, \url{https://doi.org/10.1002/nla.2276}.

\bibitem{CAA:Porcellietal}
{\sc M.~Porcelli, V.~Simoncini, and M.~Stoll}, {\em {P}reconditioning
  {PDE}-constrained optimization with $l^1$-sparsity and control constraints},
  Computers \& Mathematics with Applications, 74 (2017), pp.~1059--1075,
  \url{https://doi.org/10.1016/j.camwa.2017.04.033}.

\bibitem{COAP:PougkGond}
{\sc S.~Pougkakiotis and J.~Gondzio}, {\em {A}n interior point-proximal method
  of multipliers for convex quadratic programming}, Computational Optimization
  and Applications, 78 (2021), pp.~307--351,
  \url{https://doi.org/10.1007/s10589-020-00240-9}.

\bibitem{SIAMX:Pougketal}
{\sc S.~Pougkakiotis, J.~W. Pearson, S.~Leveque, and J.~Gondzio}, {\em {F}ast
  solution methods for convex quadratic optimization of fractional differential
  equations}, SIAM Journal on Matrix Analysis and Applications, 41 (2020),
  pp.~1443--1476, \url{https://doi.org/10.1137/19M128288X}.

\bibitem{MathOR:Qi}
{\sc L.~Qi}, {\em {C}onvergence analysis of some algorithms for solving
  nonsmooth equations}, Mathematics of Operations Research, 18 (1993),
  pp.~227--244, \url{https://doi.org/10.1287/moor.18.1.227}.

\bibitem{RobinsonMathProgStud}
{\sc S.~M. Robinson}, {\em {S}ome continuity properties of polyhedral
  multifunctions}, in Mathematical Programming at Oberwolfach, H.~König,
  B.~Korte, and K.~Ritter, eds., vol.~14 of Mathematical Programming Studies,
  Springer, Berlin, Heidelberg, 1981, pp.~206--214,
  \url{https://doi.org/10.1007/BFb0120929}.

\bibitem{MathOpRes:Rock}
{\sc R.~T. Rockafellar}, {\em {{A}ugmented {L}agrangians and applications of
  the proximal point algorithm in convex programming}}, Mathematics of
  Operations Research, 1 (1976), pp.~97--116,
  \url{https://doi.org/doi.org/10.1287/moor.1.2.97}.

\bibitem{SIAMJCO:Rock}
{\sc R.~T. Rockafellar}, {\em {M}onotone operators and the proximal point
  algorithm}, SIAM Journal on Control and Optimization, 14 (1976),
  pp.~877--898, \url{https://doi.org/10.1137/0314056}.

\bibitem{Springer:RockWets}
{\sc R.~T. Rockafellar and R.~J.~B. Wets}, {\em {V}ariational {A}nalysis},
  vol.~317 of Grundlehren der mathematischen Wissenschaften, Springer-Verlag
  Berlin Heidelberg, 1998, \url{https://doi.org/10.1007/978-3-642-02431-3}.

\bibitem{COAP:SongChenYu}
{\sc X.~Song, B.~Chen, and B.~Yu}, {\em {A}n efficient duality-based approach
  for {PDE}-constrained sparse optimization}, Computational Optimization and
  Applications, 69 (2018), pp.~461--500,
  \url{https://doi.org/10.1007/s10589-017-9951-4}.

\bibitem{COAP:Stella_etal}
{\sc L.~Stella, A.~Themelis, and P.~Patrinos}, {\em {F}orward-backward
  quasi-{N}ewton methods for nonsmooth optimization problems}, Computational
  Optimization and Applications, 67 (2017), pp.~443--487,
  \url{https://doi.org/10.1007/s10589-017-9912-y}.

\bibitem{osqp}
{\sc B.~Stellato, G.~Banjac, P.~Goulart, A.~Bemporad, and S.~Boyd}, {\em
  {OSQP}: an operator splitting solver for quadratic programs}, Mathematical
  Programming Computation, 12 (2020), pp.~637--672,
  \url{https://doi.org/10.1007/s12532-020-00179-2},
  \url{https://doi.org/10.1007/s12532-020-00179-2}.

\bibitem{SIAMOpt:Themelis_etal}
{\sc A.~Themelis, L.~Stella, and P.~Patrinos}, {\em {F}orward-backward envelope
  for the sum of two nonconvex functions: {F}urther properties and nonmonotone
  linesearch algorithms}, SIAM Journal on Optimization, 28 (2018),
  pp.~2274--2303, \url{https://doi.org/10.1137/16M1080240}.

\bibitem{AMS:Trolzsch}
{\sc F.~Tröltzsch}, {\em {O}ptimal {C}ontrol of {P}artial {D}ifferential
  {E}quations: {T}heory, {M}ethods and {A}pplications}, vol.~112 of Graduate
  Studies in Mathematics, American Mathematical Society, 2010,
  \url{https://doi.org/10.1090/gsm/112}.

\bibitem{SIAMOpt:Vander}
{\sc R.~J. Vanderbei}, {\em {S}ymmetric quasidefinite matrices}, SIAM Journal
  on Optimization, 5 (1993), pp.~100--113,
  \url{https://doi.org/10.1137/0805005}.

\bibitem{bookVapnik}
{\sc V.~N. Vapnik}, {\em {S}tatistical {L}earning {T}heory}, John Wiley \&
  Sons, New York, 1998.

\bibitem{ESAIM:GerdDaniel}
{\sc G.~Wachsmuth and D.~Wachsmuth}, {\em {C}onvergence and regularization
  results for optimal control problems with sparsity functional}, ESAIM:
  Control, Optimisation and Calculus of Variations, 17 (2011), pp.~858--886,
  \url{https://doi.org/10.1051/cocv/2010027}.

\bibitem{IPMs:WaltzMoralNocedOrban}
{\sc R.~A. Waltz, J.~L. Morales, J.~Nocedal, and D.~Orban}, {\em An interior
  algorithm for nonlinear optimization that combines line search and trust
  region steps}, Mathematical Programming, 107 (2006), pp.~391--408,
  \url{https://doi.org/10.1007/s10107-004-0560-5}.

\bibitem{JRSS:Zou}
{\sc H.~Zou and T.~Hastie}, {\em {R}egularization and variable selection via
  the elastic net}, Journal of the Royal Statistical Society: Series B
  (Statistical Methodology), 67 (2005), pp.~301--320,
  \url{https://doi.org/10.1111/j.1467-9868.2005.00503.x}.

\end{thebibliography}
\bibliographystyle{siamplain}

\end{document}